\def\arxivVersion{}
\def\mytitle{Feedback Refinement Relations\\for the Synthesis of Symbolic Controllers}

\def\myname{
\ifx\shortnames\undefined Gunther \fi
Reissig%
,
\ifx\shortnames\undefined Alexander \fi
Weber%
, and
\ifx\shortnames\undefined Matthias \fi
Rungger
}

\def\mykeywords{Discrete abstraction, symbolic model, nonlinear
system, symbolic control, automated synthesis,
robust synthesis%
\ifx\arxivVersion\undefined\relax\else;
MSC: Primary, 93B51;
Secondary, 93B52, 93C10, 93C30, 93C55, 93C57, 93C65%
\fi}
\let\submissionNote=\relax
\let\websourceNote=\relax

\ifx\FinalVersion\undefined%
\ifx\arxivVersion\undefined%
\def\DraftVersion{}
\documentclass[a4paper]{IEEEtran}
\else
\def\submissionNote{This work has been accepted for publication in the
\emph{IEEE Trans. Automat. Control}. Please refer to
\href{http://dx.doi.org/10.1109/TAC.2016.2593947}{DOI: 10.1109/TAC.2016.2593947}
for the definite publication.
}
\def\websourceNote{To reference this work, please find a {Bib\TeX}
  entry at
\href{http://www.reiszig.de/gunther/pubs/i14abs.html}{author's homepage}.
}
\documentclass[journal,a4paper,onecolumn,12pt,nofonttune]{IEEEtran}
\date{}
\fi
\else
\def\IEEEtheorems{}
\documentclass[letterpaper]{IEEEtran}
\fi

\usepackage{tikz}
\usetikzlibrary{calc,shapes,arrows,automata}
\tikzset{ block/.style    = {draw, thick, rectangle, minimum height = .8cm, minimum width = .8cm}, }
\tikzset{ my loop/.style={to path={ .. controls +(90:1) and +(130:1) .. (\tikztotarget) \tikztonodes}}}
\newcommand{\cox}{\tikz{\draw (0,0) circle (.7ex); \draw[rotate=45] (-0.45ex,-.45ex) rectangle (.45ex,.45ex); }}

\ifx\DraftVersion\undefined\relax\else%
\usepackage{ulem}
\def\emph#1{\textit{#1}}
\renewcommand\sout{\bgroup\markoverwith
{\textcolor{red}{\rule[.5ex]{2pt}{1pt}}}\ULon}
\fi

\usepackage{cite}
\usepackage{pgf}
\usepackage{psfrag}
\usepackage{xcolor}

\usepackage{svn-multi}
\svnid{$Id: i14sym.tex 957 2016-07-21 10:48:30Z lf1balwe $}

\usepackage[cmex10]{amsmath}
\let\ORGforeignlanguage\foreignlanguage
\def\foreignlanguage#1{\lowercase{\ORGforeignlanguage{#1}}}
\makeatletter
\def\MakeUppercase#1{#1}%
\def\markboth#1#2{\def\leftmark{\@IEEEcompsoconly{\sffamily}\MakeUppercase{#1}}%
\def\rightmark{\@IEEEcompsoconly{\sffamily}\MakeUppercase{#2}}}
\makeatother
\ifx\arxivVersion\undefined\relax\else%

\fi

\usepackage{GR/gunther2e}

\ifx
\hypersetup\undefined\def\href#1#2{\textbf{#2}}
\else
\hypersetup{
pdftitle={\mytitle},
pdfauthor={Gunther Reissig, Alexander Weber, Matthias Rungger, http://www.reiszig.de/gunther/},%
pdfsubject={\submissionNote \websourceNote},
pdfkeywords={\mykeywords}
}
\fi

\theoremstyle{hyp*}

\makeatletter\def\p@condSR#1{(SR)}\makeatother
\newtheorem{condASR}{ASR}
\makeatletter\def\p@condASR#1{(ASR)}\makeatother
\newtheorem{condC}{C}
\makeatletter\def\p@condC#1{(C)}\makeatother
\newtheorem{condZ}{Z}
\makeatletter\def\p@condZ#1{(Z)}\makeatother

\usepackage{multibib} 
\newcites{review}{References}

\begin{document}\bstctlcite{IEEEtranBSTCTL_AuthorNoDash}%
\makeatletter
\renewenvironment{proof}[1][\proofname]{\par
  \pushQED{\qed}%
  \normalfont \topsep6\p@\@plus6\p@\relax
  \trivlist
  \item[\hskip\labelsep
        \itshape
    #1\@addpunct{.}]\ignorespaces
}{%
  \popQED\endtrivlist\@endpefalse
}
\markboth{\def\shortnames{}\myname\hspace*{\fill}\def\\{ }\mytitle\hspace*{\fill}\@date\hspace*{\fill}\ifx\DraftVersion\undefined\relax\else\hspace*{\fill}svn: \svnrev\fi\hspace*{\fill}}%
{\def\shortnames{}\myname\hspace*{\fill}\def\\{ }\mytitle\hspace*{\fill}\@date\hspace*{\fill}\ifx\DraftVersion\undefined\relax\else\hspace*{\fill}svn: \svnrev\fi\hspace*{\fill}}%
\makeatother

\title{\mytitle}

\author{\myname%
\thanks{%
\ifx\DraftVersion\undefined\relax\else%
Corresponding author: G. Reissig.\newline
\fi
G. Reissig and A. Weber are with the
University of the Federal Armed Forces Munich,
Dept. Aerospace Eng.,
Chair of Control Eng. (LRT-15),
D-85577 Neubiberg (Munich),
Germany,
\ifx\DraftVersion\undefined%
\url{http://www.reiszig.de/gunther/}%
\else%
\url{gunther@reiszig.de}, \url{A.Weber@unibw.de}%
\fi%
}%
\thanks{%
M. Rungger is with the Hybrid Control Systems Group at the Department of
Electrical and Computer Engineering at the Technical University of Munich, 80333 Munich, Germany%
\ifx\DraftVersion\undefined.\else%
, \url{matthias.rungger@tum.de}%
\fi
}%
\thanks{This work has been supported by the German Research Foundation (DFG) under grants no. RE 1249/3-2 and RE 1249/4-1.%
\ifx\arxivVersion\undefined\relax\else%
{} \submissionNote{} \websourceNote%
\fi%
}
}

\maketitle

\begin{abstract}
\noindent
We present an abstraction and refinement methodology
for the automated controller synthesis to enforce
general predefined specifications. The
designed controllers require
quantized (or symbolic) state information only and can be interfaced
with the system via a static quantizer.
Both features are
particularly important with regard to any practical implementation of
the designed controllers
and, as we prove, are characterized by the existence
of a feedback refinement relation between plant and abstraction.
Feedback refinement relations are a novel concept
introduced in this paper.
Our work builds on a general notion of system with set-valued dynamics
and possibly non-deterministic quantizers to permit the synthesis
of controllers that robustly, and provably, enforce the
specification in the presence of various types of uncertainties and
disturbances.
We identify a class of abstractions that is
canonical in a well-defined sense, and
provide a method to efficiently compute canonical abstractions.
We demonstrate the practicality
of our approach on two examples.
\end{abstract}

\ifx\arxivVersion\undefined\else%
\fi

\section{Introduction}
\label{s:intro}
A common approach to engineer reliable, robust, high-integrity
hardware and software systems that are deployable in safety-critical
environments, is the application of formal verification techniques to
ensure the correct, error-free implementation of some given formal
specifications.
Typically, the verification phase is executed as a
distinct step after the  design phase, e.g. \cite{Palnitkar03}.
In case that the
system fails to satisfy the specification, it is the engineer's
burden to identify the fault, adjust the system accordingly and return to the
verification phase. A more appealing approach, especially in the
context of intricate, complex dynamical systems,
is to merge the design and verification phase and utilize
automated correct-by-construction formal synthesis procedures,
e.g. \cite{Tabuada09}.
In our treatment of controller design problems %
we follow the latter approach.
That is, given a mathematical system
description and a formal specification which expresses the desired
system behavior, we seek to synthesize a
controller that provably enforces the specification on the system.
Subsequently, we often refer to the
system that is to be controlled as the \begriff{plant}.

For finite systems, which are described by transition systems
with finite state, input and output alphabets, there exist a number of
automata-theoretic schemes,
known under the label of \begriff{reactive synthesis},
to algorithmically synthesize
controllers that enforce complex specifications, possibly formulated
in some temporal logic, see
e.g.~\cite{EmersonClarke82,PnueliRosner89,Vardi95,Tabuada09,BloemJobstmanPitermanPnueliSaar12}.

Those methods have been extended to infinite systems
within an abstraction and refinement
framework,
e.g.~\cite{Tabuada09,i11abs,GrueneJunge07,KreisselmeierBirkholzer94,Girard13,%
YordanovTumovaCernaBarnatBelta12,DallalColomboDelVecchioLafortune13,LiuOzay14,%
ZamaniPolaMazoTabuada10,RunggerStursberg12,%
KoutsoukosAntsaklisStiverLemmon00,%
i13absocc,GirardPappas07b,MoorSchmidtWittmann11,Girard14c},
which roughly proceeds in three steps.
In the first step, the concrete infinite system (together with the
specification) is lifted to an abstract
domain where it is substituted by a finite system, which is often referred to
as \begriff{abstraction} or \begriff{symbolic model}. In the second step,
an auxiliary problem on the abstract domain (``abstract problem'')
is solved using one of the previously mentioned methods
for finite systems. In the third step, the
controller that has been synthesized for the abstraction
is refined to the concrete system.

The correctness of this controller design concept is usually ensured
by relating the concrete system with its abstraction in terms of a
system relation.
The most common approaches are based on
\begriff{(alternating) (bi-)simulation relations}
and approximate variants thereof
\cite{Tabuada09}.
In this work, we address two shortcomings of the abstraction and
refinement process based on simulation relations and related concepts. 
The first shortcoming, which we refer to as the \begriff{state
  information issue}, results from the fact that the
refined controller
requires the exact state information of the concrete system.
However, usually, the exact state is not known and only quantized (or symbolic) state
information is available, which constitutes a major obstacle to the
practical implementation of the synthesized controllers.
The second issue refers to the huge amount of dynamics added
to the abstract controller in the course of its refinement, so that,
effectively, the refined controller contains the abstraction as a
building block.
Given the fact that an abstraction may very well comprise
millions of states and billions of
transitions~\cite{i11abs,ZamaniPolaMazoTabuada10}, an implementation
of the refined controller is often too expensive to be practical.
We refer to this problem as the \begriff{refinement
complexity issue}. We
illustrate both issues by
examples in Section~\ref{s:pitfalls}. See also \cite{i14symc}.

In this paper, we propose a novel notion of system relation, termed
\begriff{feedback refinement relation}, to resolve both issues.
If the concrete system is related with the abstraction via a feedback
refinement relation, then, as we shall show, the abstract controller
can be connected to the plant via a static quantizer only,
irrespective of the particular specification we seek to enforce on the
plant. See \ref{f:ClosedLoop_FRR}.
\begin{figure}[ht]
\begin{center}
\psfrag{input}[r][r]{input}
\psfrag{converter}[][]{converter}
\psfrag{plant}[][]{plant}
\psfrag{quantizer}[][]{quantizer}
\psfrag{refined controller}[][]{refined controller}
\psfrag{controller}[][]{controller}
\psfrag{interface}[][]{interface}
\psfrag{state}[l][l]{state}
\psfrag{abstraction}[][]{abstraction}
\psfrag{abstract}[][]{abstract}
\psfrag{refined}[][]{refined}
\includegraphics[width=0.6\linewidth]{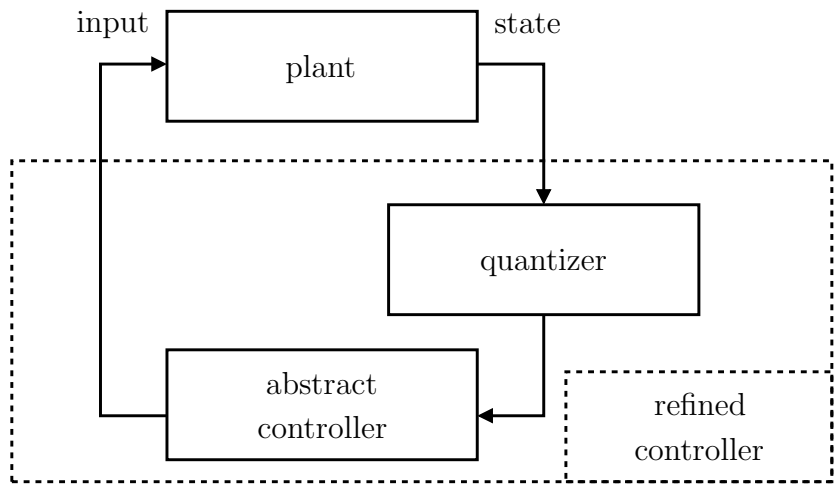}%
\end{center}
\caption{
Closed loop resulting from the abstraction and refinement approach based
on feedback refinement relations, proposed in this paper.
}\label{f:ClosedLoop_FRR}
\end{figure}
Moreover, the existence of a feedback refinement relation between plant and
abstraction is not only sufficient to ensure the simple structure of
the closed loop in \ref{f:ClosedLoop_FRR}, but in fact also
necessary.

Our work builds on a general notion of system with set-valued
dynamics and possibly non-deterministic quantizers. This is
particularly useful to model various types of disturbances, including
plant uncertainties, input disturbances and state measurement errors.
We demonstrate how to account for those perturbations in our framework so that the
synthesized controllers robustly enforce the specification.

In general, abstractions over-approximate the plant behavior, and so
their practical use will depend on the accuracy of the approximation
that can be achieved
by actual computational methods;
see the discussion in \cite[Sect.~I]{i11abs}.
In this regard, we show
that the set membership relation together with
an abstraction whose state alphabet is a cover of the concrete state
alphabet is canonical in a well-defined sense, and provide a method to
compute canonical abstractions of perturbed nonlinear
sampled systems.
The practicality of the approach is demonstrated on
two examples -- a path planning problem for an autonomous vehicle and
an aircraft landing maneuver.

{\bf Related Work.}
\label{pageref:review:item3}
Feedback refinement relations are based on the common principle of 
``accepting more inputs and generating fewer outputs'' that is often
encountered in component-based design methodologies, e.g.~contract-based
design~\cite{BenvenutiFerrariMazziSangiovanniVincentelli08} and interface
theories~\cite{TripakisLicklyHenzingerLee11}. Those theories are
usually developed in a purely behavioral setting, see e.g.~\cite{BenvenutiFerrariMazziSangiovanniVincentelli08,TripakisLicklyHenzingerLee11,
MoorSchmidtWittmann11}, and are therefore not immediately applicable
in our framework which is based on stateful systems.
This class of systems contains a great variety of system
descriptions, including common models like transition systems~\cite{BaierKatoen08,Tabuada09} as well as
discrete-time control systems~\cite{Sontag98}.

There exist a number of abstraction-based controller synthesis
methods, based on stateful systems,
that do not suffer from the state information issue nor from the
refinement complexity
issue~\cite{KreisselmeierBirkholzer94,GrueneJunge07,i11abs,YordanovTumovaCernaBarnatBelta12,Girard13,DallalColomboDelVecchioLafortune13,LiuOzay14}.
However, none of those approaches offers necessary and sufficient
conditions for the controller refinement procedure to be free of the
mentioned issues.  In addition, the majority of these works are
tailored to certain types of specifications or systems. Specifically,
simple safety and reachability problems are considered
in~\cite{DallalColomboDelVecchioLafortune13,Girard13} and
\cite{KreisselmeierBirkholzer94,GrueneJunge07,i11abs,Girard13},
respectively,
while
\cite{YordanovTumovaCernaBarnatBelta12,Girard13,DallalColomboDelVecchioLafortune13}
is limited to piecewise affine, incrementally stable, and simple
integrator dynamics, respectively. Moreover,
plants are assumed to be non-blocking in
\cite{KreisselmeierBirkholzer94,GrueneJunge07,i11abs,YordanovTumovaCernaBarnatBelta12,Girard13,DallalColomboDelVecchioLafortune13,LiuOzay14}.
In contrast, our framework covers stateful systems
with general, set-valued dynamics, including transitions systems and
discrete-time control systems as special cases. We allow systems to
be blocking,
and
any linear time property can serve as a
specification.

A class of methods known under the label of
\begriff{hierarchical control} 
are similar in spirit to abstraction-based methods in that they
synthesize discrete controllers using finite-state models derived from
concrete control problems, e.g.
\cite{CainesWei98,FainekosKressGazitPappas05,HabetsCollinsvanSchuppen06}.
However,
\label{review:item3:text}\label{review:item2:text}
the finite-state models in
\cite{FainekosKressGazitPappas05,HabetsCollinsvanSchuppen06}
are not abstractions in the usual sense, in that they approximate the
behavior of an interconnection of the plant with low-level
controllers, rather than the behavior of the plant itself.
In~\cite{CainesWei98} one is required to derive a quantizer in
accordance with the exact plant dynamics, and to verify rather complex
system properties.
Moreover, those hierarchical schemes require exact state information or, in 
\label{review:item5:text}
the case of linear output feedback~\cite{HabetsCollinsvanSchuppen12}, require
exact output information, and are unable to account for quantized
or perturbed measurements.
Additionally, for general nonlinear plants,
all of the aforementioned approaches
require the synthesis of low-level controllers to enforce
a high-level plan, which is considered as an open
problem~\cite{KressGazitWongpiromsarnTopcu11} and current
solutions exist
only for rather restrictive classes of
\label{review:item4:text}
systems~\cite{HabetsCollinsvanSchuppen12,HelwaBroucke13,HelwaCaines15}. 
In contrast, the refinement step in our approach
is completely independent of the plant dynamics and
\label{review:item6:text}
does not involve the design of low-level controllers.
\looseness-1

For any of the aforementioned approaches, often a lack of robustness
further restricts the applicability of the methods.  For example,
\cite{KreisselmeierBirkholzer94,YordanovTumovaCernaBarnatBelta12,Girard13} do not cover uncertainties
in plant dynamics, while
in~\cite{GrueneJunge07,YordanovTumovaCernaBarnatBelta12,Girard13,CainesWei98,FainekosKressGazitPappas05,HabetsCollinsvanSchuppen06}
the quantizer is assumed to be deterministic which mandates the state
measurement to be precise, without any error; see
Section~\ref{ss:Uncertainties}.

Similarly to our work, the synthesis scheme
in~\cite{LiuOzay14} introduces a novel system relation.
However, in contrast to the theory in~\cite{LiuOzay14}, feedback
refinement relations do not rely on a metric of the state
alphabet, which is crucial in establishing the necessity as well as
the canonicity result. Likewise, the authors of \cite{LiuOzay14}
consider perturbations, but assume that the effect of these
perturbations is given as level sets of a metric.

In addition to a general synthesis framework, we present
a method to construct abstractions
of perturbed nonlinear control systems. The abstractions are based on a
cover of the state alphabet by non-empty compact hyper-intervals and the
over-approximation of attainable sets of those hyper-intervals under
the system dynamics.
While the use of attainable
sets for the construction of abstractions is a well-known
concept~\cite{GrueneJunge07,i11abs,ZamaniPolaMazoTabuada10,RunggerStursberg12},
none of the aforementioned works accounts for uncertainties or perturbations.
Moreover, while our method to over-approximate attainable sets is
similar to those in \cite{ZamaniPolaMazoTabuada10,RunggerStursberg12}
in that it is based on a growth bound, we present several extensions
that render the approach more efficient.

To summarize, our contribution is threefold. First, we introduce feedback
refinement relations as a novel means to synthesize symbolic
controllers. We show that feedback refinement relations are necessary and sufficient
for the controller refinement that solves the state information issue
and the refinement complexity issue.
Our theory applies to a more general class of synthesis problems than
previous research that addresses the mentioned issues, and in
particular, any linear time property can serve as a specification.
Second, our work permits the synthesis of controllers that robustly,
and provably, enforce the specification in presence of various
uncertainties and disturbances.
Third, we identify a class of canonical abstractions and present a
method to compute such abstractions.  Our construction improves known
methods in several directions and thereby, as we demonstrate by some
numerical examples, facilitates a more efficient computation of
abstractions of perturbed nonlinear control systems.

Some of the results we present have been announced in \cite{i14symc}.

\section{Notation}
\label{s:prelims}

The relative complement of the set $A$ in the set $B$ is denoted by
$B \setminus A$.
$\mathbb{R}$, $\mathbb{R}_+$, $\mathbb{Z}$ and $\mathbb{Z}_{+}$ denote the sets of
real numbers, non-negative real numbers, integers and non-negative integers, respectively,
and $\mathbb{N} = \mathbb{Z}_{+} \setminus \{ 0 \}$. We adopt the
convention that $\pm \infty + x = \pm \infty$ for any
$x \in \mathbb{R}$.
$\intcc{a,b}$, $\intoo{a,b}$,
$\intco{a,b}$, and $\intoc{a,b}$
denote closed, open and half-open, respectively,
intervals with end points $a$ and $b$.
$\intcc{a;b}$, $\intoo{a;b}$,
$\intco{a;b}$, and $\intoc{a;b}$ stand for discrete intervals, e.g.
$\intcc{a;b} = \intcc{a,b} \cap \mathbb{Z}$
and
$\intco{0;0} = \emptyset$.

In $\mathbb{R}^n$, the relations $<$, $\leq$, $\geq$, $>$ are defined
component-wise, e.g.
$a < b$ iff $a_i < b_i$ for all $i \in \intcc{1;n}$.

$f \colon A \rightrightarrows B$ denotes a \begriff{set-valued map} of
$A$ into $B$, whereas $f \colon A \to B$ denotes an ordinary map; see
\cite{RockafellarWets09}. If $f$ is
set-valued, then $f$ is \begriff{strict} and \begriff{single-valued} if
$f(a) \not= \emptyset$ and $f(a)$ is a singleton,
respectively, for every $a$.
The restriction of $f$ to a subset $M \subseteq A$ is denoted
$f|_{M}$.
Throughout the text, we denote the identity map
$X \to X \colon x \mapsto x$ by $\id$. The domain of definition $X$
will always be clear form the context.

We identify set-valued maps
$f \colon A \rightrightarrows B$ with binary relations on
$A \times B$, i.e., $(a,b) \in f$ iff $b \in f(a)$.
Moreover, if $f$ is single-valued, it
is identified with an ordinary map $f \colon A \to B$.
The inverse mapping $f^{-1} \colon B \rightrightarrows A$ is defined
by $f^{-1}(b) = \Menge{a \in A}{b \in f(a)}$, and $f \circ g$ denotes
the composition of $f$ and $g$, $(f \circ g)(x) = f(g(x))$.

The set of maps $A \to B$ is denoted $B^A$, and the set of all signals
that take their values in $B$ and are defined on intervals of the form
$\intco{0;T}$ is denoted $B^{\infty}$,
$B^{\infty} = \bigcup_{T \in \mathbb{Z}_{+}\cup \{\infty\}} B^{\intco{0;T}}$.

\section{Plants, Controllers, and Closed Loops}
\label{s:Systems}

\subsection{Systems}
\label{ss:SystemsBehaviorsErrorCompletion}

We consider dynamical systems of the form
\begin{IEEEeqnarray}{c}\label{e:simpleSys}
\begin{IEEEeqnarraybox}[][c]{rCl}
x(t+1)&\in&F(x(t),u(t))\\
  y(t)&\in&H(x(t),u(t)).
\end{IEEEeqnarraybox}
\end{IEEEeqnarray}
The motivation to use a set-valued transition function $F$ and a
set-valued output function $H$ in our system description, originates from
the desire to describe disturbances and other kinds of
non-determinism in a unified and concise manner.
This description is also sufficiently expressive to model the
plant and the controller, but unfortunately leads to subtle issues
with interconnected systems. Consider e.g.~the serial composition in
\ref{f:serial}, where
$F_i \colon X_i \times U_i \rightrightarrows X_i$,
$X_1 = U_1 = \{0\}$, $X_2 = U_2 = \{0,1\}$,
$Y_2 = \{ a, b, c \}$,
$F_1(0,0) = \{0\}$, $H_1(0,0) = U_2$,
and $F_2$ and
$H_2 \colon X_2 \times U_2 \rightrightarrows Y_2$
are given as follows:
$F_2(1,0) = F_2(0,1) = \{ 0 \}$,
$F_2(0,0) = F_2(1,1) = \{ 1 \}$,
$H_2(0,0) = H_2(1,0) = \{ a \}$,
$H_2(0,1) = \{ b \}$, and
$H_2(1,1) = \{ c \}$.
To recover the behavior at the terminals $u_1$ and $y_2$ with
a system of the form~\ref{e:simpleSys}, we let
$X = X_1 \times X_2$,
$F \colon X \times U_1 \rightrightarrows X$ and
$H \colon X \times U_1 \rightrightarrows Y_2$.
As $Y_2$ contains more elements than $X \times U_1$, which can
all appear in $y_2$, the map $H$ must be multi-valued, which in turn
implies that the following property of the composed system in
\ref{f:serial} cannot be retained:
Between any two appearances of $b$ in $y_2$ there are an even number of
$a$'s, and between any appearance of $b$ and any appearance of $c$
there are an odd number of $a$'s.
It follows that the class of systems of the form~\ref{e:simpleSys} is
not closed under interconnection, given the natural constraint that
the state alphabet of the composed system equals the product of the
state alphabets of the individual systems.
To circumvent this problem we
\begin{figure}[h]
\centering
\begin{tikzpicture}[node distance=1.8cm, >=latex]
\def\dx{0.15}
\draw node at (0,0) [block] (F1)  {$F_1$};
\draw node[right of=F1] [block] (H1)  {$H_1$};

\draw[<-] ($(F1.west)-(0,\dx)$)  -- node[below,near end] {$u_1$} ++(-1.25,0);
\draw[->] ($(F1.east)+(0,\dx)$)  -- node[above,near end] {$x_1$} ($(H1.west)+(0,\dx)$);
\draw[->] ($(F1.east)+(.5,\dx)$) -- ++(0,0.4) -- ++(-1.75,0) |- ($(F1.west)+(0,\dx)$);
\draw[->] ($(F1.west)-(0,\dx)$) ++(-.5,0) -- ++(0,-0.4) -- ++(1.75,0) |- ($(H1.west)-(0,\dx)$);
\draw[->] node at ($(F1.west)-(.5,\dx)$) {\textbullet};
\draw[->] node at ($(F1.east)+(.5,\dx)$) {\textbullet};

\draw node at ($(F1.east)+(0.25,\dx)$) {$\sslash$};

\draw[dashed] ($(F1.west)-(.65,.75)$) rectangle ++(3.4,1.5);
\draw node at (4,0) [block] (F2)  {$F_2$};
\draw node[right of=F2] [block] (H2)  {$H_2$};

\draw[->] (H2)  -- node[below,near end] {$y_2$} ++(1.25,0);
\draw[->] ($(F2.east)+(0,\dx)$)  -- node[above,near end] {$x_2$} ($(H2.west)+(0,\dx)$);
\draw[->] ($(F2.east)+(.5,\dx)$) -- ++(0,0.4) -- ++(-1.75,0) |- ($(F2.west)+(0,\dx)$);
\draw[->] ($(F2.west)-(0,\dx)$) ++(-.5,0) -- ++(0,-0.4) -- ++(1.75,0) |- ($(H2.west)-(0,\dx)$);
\draw[->] node at ($(F2.west)-(.5,\dx)$) {\textbullet};
\draw[->] node at ($(F2.east)+(.5,\dx)$) {\textbullet};

\draw node at ($(F2.east)+(0.25,\dx)$) {$\sslash$};

\draw[dashed] ($(F2.west)-(.65,.75)$) rectangle ++(3.4,1.5);
\draw[->]
($(H1.east)-(0,\dx)$) -- 
node[below, label={[shift={(-.25,-.5)}]{$y_1$}}]{} 
node[below, label={[shift={(-.25,0)}]{$u_2$}}]{} 
($(F2.west)-(0,\dx)$);

\end{tikzpicture}
\caption{Serial composition of two dynamical systems of the form
  \ref{e:simpleSys}. The symbol
$\sslash$ denotes a delay.}\label{f:serial}
\end{figure}
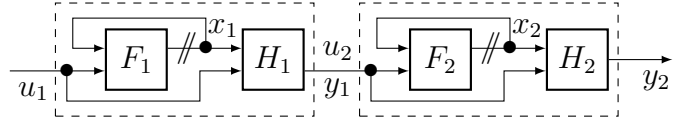
consider a slightly more general form of system dynamics given by
\begin{subequations}
\label{e:System}
\begin{align}
\label{e:System:dyn}
x(t+1) &\in F( x(t), v(t) ),\\
\label{e:System:inout}
( y(t), v(t) ) &\in H( x(t), u(t) ),
\end{align}
\end{subequations}
where $v$ is an \begriff{internal variable}. 
We formalize the notion of system as follows.

\begin{definition}
\label{def:System}
A \begriff{system} is a septuple
\begin{equation}
\label{e:def:System}
S = (X, X_0, U, V, Y, F, H),
\end{equation}
where
$X$, $X_0$, $U$, $V$ and $Y$ are nonempty sets,
$X_0 \subseteq X$,
$H \colon X \times U \rightrightarrows Y \times V$ is strict,
and $F \colon X \times V \rightrightarrows X$.
\\
A quadruple
$(u,v,x,y) \in U^{\intco{0;T}} \times V^{\intco{0;T}} \times X^{\intco{0;T}} \times Y^{\intco{0;T}}$
is a \begriff{solution} of the system \ref{e:def:System} (on
$\intco{0;T}$, starting at $x(0)$)
if $T \in \mathbb{N} \cup \{ \infty \}$,
\ref{e:System:dyn} holds for all $t \in \intco{0;T-1}$,
\ref{e:System:inout} holds for all $t \in \intco{0;T}$, and
$x(0) \in X_0$.
\end{definition}
The internal variables allow us to introduce the constraint $u_2=y_1$ imposed by the
composition in~\ref{f:serial} and recover the behavior of the serial composed
system with a system of the form~\ref{e:def:System} given 
by $X = X_0 = \{ 0, 1 \}$, $U = \{ 0 \}$, 
$V=Y=\{a,b,c\}$ with $F(0,a)=F(1,c)=\{1\}$, $F(1,a)=F(0,b)=\{0\}$ and
$H(0,0)=\{(a,a),(b,b)\}$, $H(1,0)=\{(a,a),(c,c)\}$.

We call the sets $X$, $X_0$, $U$, $V$, and $Y$
the \begriff{state}, \begriff{initial state}, \begriff{input}, \begriff{internal variable}, and
\begriff{output alphabet}, respectively.
The functions
$F$ and $H$ are, respectively, the \begriff{transition function}
and the \begriff{output function} of~\ref{e:def:System}.
We call the system \ref{e:def:System}
\begin{enumerate}
  \item \begriff{autonomous} if $U$ is a singleton;
  \item \begriff{static} if $X$ is a singleton;
\label{review:item1:text}
  \item \begriff{Moore} if the output does not depend on the input, i.e.,
$(y,v) \in H(x,u) \wedge u' \in U$
\implies
$\exists_{v'} (y,v') \in H(x,u')$;
\footnote{The notation $\exists_s A$ reads as ``there exists $s$ such
that the statement $A$ holds''.}
\item
\begriff{simple}, if
$U = V$, $X = Y$, $H = \id$, and
all states are admissible as initial states, i.e., $X = X_0$.
\end{enumerate}
We assume throughout that the plant is given by a
simple system, which restricts our theory to that class of plants.

\subsection{System composition}
\label{ss:system_composition}

In the following, we define the serial and feedback composition of
two systems. We start with the serial composition.

\begin{definition}
\label{def:serial}
Let
$S_i = (X_i,X_{i,0},U_i,V_i,Y_i,F_i,H_i)$ be systems,
$i \in \{1,2\}$, and assume that
$Y_1 \subseteq U_2$. Then $S_1$ is \begriff{serial composable} with $S_2$, and 
the \begriff{serial composition} of $S_1$ and $S_2$,
denoted $S_2 \circ S_1$, is the septuple
\[
(X_{12}, X_{1,0} \times X_{2,0}, U_1, V_{12}, Y_{2}, F_{12}, H_{12}),
\]
where
$X_{12} = X_1 \times X_2$,
$V_{12} = V_1 \times V_2$, 
$F_{12} \colon X_{12} \times V_{12} \rightrightarrows X_{12}$ and
$H_{12} \colon X_{12} \times U_1 \rightrightarrows Y_{2} \times V_{12}$ satisfy
\ifCLASSOPTIONonecolumn
\begin{IEEEeqnarray*}{rCl}
F_{12}(x,v) &=& F_1(x_1,v_1) \times F_2(x_2,v_2),\\
H_{12}(x,u_1) &=&
\{ (y_2,v) \mid\exists_{y_1}
(y_1,v_1) \in H_1(x_1,u_1)
\wedge
(y_2,v_2) \in H_2(x_2,y_1)
\}.
\popQED
\end{IEEEeqnarray*}
\else
\begin{IEEEeqnarray*}{rCl}
F_{12}(x,v) &=& F_1(x_1,v_1) \times F_2(x_2,v_2),\\
H_{12}(x,u_1) &=&
\{ (y_2,v) \mid\exists_{y_1}
(y_1,v_1) \in H_1(x_1,u_1)
\\
&&\hphantom{\{ (y,v) |\exists_{y_1}(y_1,v_1)}
\wedge
(y_2,v_2) \in H_2(x_2,y_1)
\}.
\popQED
\end{IEEEeqnarray*}
\fi
\end{definition}
We readily see that the output function $H_{12}$
is strict which implies that $S_2 \circ S_1$ is a system.
We use the serial composition mainly to describe the
interconnection of  
an input quantizer $Q \colon U' \rightrightarrows U$
or a state quantizer
$Q\colon X \rightrightarrows X'$ with a system $S$
of the form \ref{e:def:System}.
We assume that $Q$ is strict and interpret the quantizer as a static
system with strict transition function. Suppose that $U'$ is a non-empty set, 
then the serial composition $S \circ Q$ of $Q$ and $S$ is defined by
\[
S \circ Q
=
(X, X_0, U', V, Y, F, H'),
\]
where $H' \colon X \times U' \rightrightarrows Y \times V$ takes the form
$H'(x,u') = H(x,Q(u'))$.
Now suppose that $S$ is
simple,
then we may interpret $Q \colon X \rightrightarrows X'$
as a measurement map that yields a quantized version
of the state of the system $S$.
This situation is modeled by the serial composition $Q \circ S$ of $S$ and $Q$,
\[
Q \circ S
=
(X, X, U, U, X', F, H'),
\]
where $H'$ takes the form $H'(x,u) = Q(x) \times \{ u \}$.

We turn our attention to the feedback composition of two systems as
illustrated in~\ref{fig:ClosedLoopMooreMealy}.

\begin{definition}
\label{def:ClosedLoop}
Let
$S_i = (X_i,X_{i,0},U_i,V_i,Y_i,F_i,H_i)$ be systems,
$i \in \{1,2\}$, and assume that
$S_2$ is Moore,
$Y_2 \subseteq U_1$ and
$Y_1 \subseteq U_2$, and that the following condition holds:
\begin{condZ}
\label{e:def:ClosedLoop:composable}
If $(y_2,v_2) \in H_2(x_2,y_1)$, $(y_1,v_1) \in H_1(x_1,y_2)$ and
$F_2(x_2,v_2) = \emptyset$, then
$F_1(x_1,v_1) = \emptyset$.
\end{condZ}
Then $S_1$ is \begriff{feedback composable} with $S_2$, and 
the \begriff{closed loop} composed of $S_1$ and $S_2$,
denoted $S_1 \times S_2$, is the septuple
\[
(X_{12}, X_{1,0} \times X_{2,0}, \{0\}, V_{12}, Y_{12}, F_{12}, H_{12}),
\]
where
$X_{12} = X_1 \times X_2$,
$V_{12} = V_1 \times V_2$,
$Y_{12} = Y_1 \times Y_2$, and
$F_{12} \colon X_{12} \times V_{12} \rightrightarrows X_{12}$ and
$H_{12} \colon X_{12} \times \{0\} \rightrightarrows Y_{12} \times V_{12}$
satisfy
\ifCLASSOPTIONonecolumn
\begin{IEEEeqnarray*}{rCl}
F_{12}(x,v) &=& F_1(x_1,v_1) \times F_2(x_2,v_2),\\
H_{12}(x,0) &=&
\{ (y,v) |
(y_1,v_1) \in H_1(x_1,y_2)
\wedge
(y_2,v_2) \in H_2(x_2,y_1)
\}.
\popQED
\end{IEEEeqnarray*}
\else
\begin{IEEEeqnarray*}{rCl}
F_{12}(x,v) &=& F_1(x_1,v_1) \times F_2(x_2,v_2),\\
\notag
H_{12}(x,0) &=&
\{ (y,v) |
(y_1,v_1) \in H_1(x_1,y_2)
\\&&\hphantom{\{ (y,v) | (y_1,v_1)}
\wedge
(y_2,v_2) \in H_2(x_2,y_1)
\}.
\popQED
\end{IEEEeqnarray*}
\fi
\end{definition}

\begin{figure}[h]
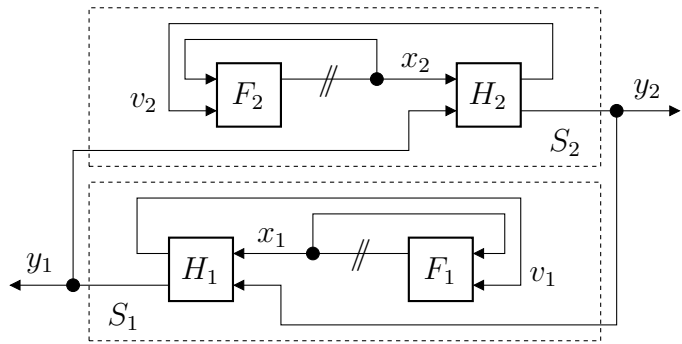

\begin{center}
\psfrag{F1}[][]{$F_1$}
\psfrag{F2}[][]{$F_2$}
\psfrag{H1}[][]{$H_1$}
\psfrag{H2}[][]{$H_2$}
\psfrag{y1}[][]{$y_1$}
\psfrag{y2}[][]{$y_2$}
\psfrag{x1}[][]{$x_1$}
\psfrag{x2}[][]{$x_2$}
\psfrag{v1}[l][l]{$v_1$}
\psfrag{v2}[r][r]{$v_2$}
\psfrag{S1}[l][l]{$S_1$}
\psfrag{S2}[r][r]{$S_2$}
\psfrag{//}[][]{$\sslash$}
\pgfdeclareimage[width=\ifCLASSOPTIONonecolumn.49\fi\linewidth]{ClosedLoopMooreMealy}{figures/ClosedLoop.MooreMealy.InputFilter}%
\noindent
\pgfuseimage{ClosedLoopMooreMealy}
\end{center}
\caption{\label{fig:ClosedLoopMooreMealy}
Closed loop $S_1 \times S_2$ of systems $S_1$ and $S_2$ according to
Definition \ref{def:ClosedLoop}, in which
the system $S_2$ is required to be Moore.
}
\end{figure}

The requirement \ref{e:def:ClosedLoop:composable}, which has its
analog in the theory developed in \cite{Tabuada09}, is particularly
important and will be needed later to ensure that if the concrete closed
loop is non-blocking, then so is the abstract closed loop.
The assumption that $S_2$ is additionally Moore is common
\cite{Vidyasagar81} and ensures that the closed loop does not contain
a delay free cycle.
We emphasize that
we avoid the assumption
that the controller is allowed to set the initial state of the plant,
as appears e.g. in~\cite{Tabuada09}.

We conclude this section with a proposition that we use in several
proofs throughout the paper.

\begin{proposition}
\label{prop:ClosedLoop:0}
Let $S_1$ be feedback composable with $S_2$,
and let $T \in \mathbb{N} \cup \{\infty\}$.
Then the closed loop $S_1 \times S_2$ is an autonomous Moore system,
and
$(0,v,x,y)$ is a solution of $S_1 \times S_2$ on $\intco{0;T}$ iff
$(y_2,v_1,x_1,y_1)$ is a solution of $S_1$ on $\intco{0;T}$ and
$(y_1,v_2,x_2,y_2)$ is a solution of $S_2$ on $\intco{0;T}$.
\end{proposition}

\begin{proof}
We claim that $H_{12}$ is strict. Indeed, assume that $x \in X_{12}$
and $a \in Y_1$. Since $H_1$ and $H_2$ are both strict, there exist
$(y_2,b) \in H_2(x_2,a)$ and $(y_1,v_1) \in H_1(x_1,y_2)$. Then there
exists $v_2$ satisfying $(y_2,v_2) \in H_2(x_2,y_1)$ as $S_2$ is
Moore, and so $(y,v) \in H_{12}(x,0)$.
This proves our claim.
The
remaining requirements in Definition \ref{def:System} are clearly
satisfied, which shows that $S_1 \times S_2$ is a system, and that
system is autonomous, and hence, Moore.
The claim on the solutions of $S_1 \times S_2$ is straightforward to
prove using Definitions \ref{def:System} and \ref{def:ClosedLoop}.
\end{proof}

\section{Motivation}
\label{s:pitfalls}

In this section, we provide two examples that 
demonstrate the state information issue and the refinement complexity
issue, which have led to the
development of the novel notion of feedback refinement relation.
Both examples show that the
drawbacks do not depend on the specific refinement technique, but
are \begriff{intrinsic} to the use of alternating (bi)simulation
relations, bisimulation relations and their approximate variants.

Let us consider two
systems $S_1$ and $S_2$ and
two controllers $C_1$ and $C_2$,
\begin{align*}
S_i
&=
( X_i, X_i, U, U, Y, F_i, H_i ),\\
C_i
&=
( X_{c,i}, X_{c,i,0}, Y, V_{c,i}, U, F_{c,i}, H_{c,i} ),
\end{align*}
in which we assume that the transition functions of the four systems
are all strict, that $X_i \subseteq Y$, and that
$H_i(x,u) = \{ (x,u) \}$ for all $(x,u) \in X_i \times U$.
We
readily see that the controller $C_i$ is feedback composable with the
system $S_i$, $i\in \{1,2\}$. 
Subsequently, we interpret $S_1$ as
the concrete system and $S_2$ as its abstraction.

Let \mbox{$Q\subseteq X_1\times X_2$} be a strict relation.
Then $Q$ is an
\begriff{alternating simulation relation}
from $S_1$ to $S_2$ 
if the following holds
for every pair $(x_1, x_2) \in Q$:
\begin{condASR}
\label{cond:ASR}
If $u_2 \in U$, then there exists $u_1 \in U$ such that the condition
\begin{equation}
\label{e:SR_ASR_condition}
\emptyset
\not=
Q(x_1') \cap F_2(x_2,u_2)
\end{equation}
holds for every $x_1' \in F_1(x_1, u_1)$.
\end{condASR}
Note that usually there is an additional condition on outputs of
related states, which here would have required the notion of
approximate rather than ordinary alternating simulation relation
\cite[Def.~9.6]{Tabuada09}. Since that subtlety is not essential
to our discussion, we omit it here in favor of a
clearer presentation.

As already mentioned, alternating simulation relations are often used
to prove the correctness of a particular abstraction-based controller
design procedure.
The very center of any such argument is the
reproducibility of the system behavior of the concrete closed loop
$C_1\times S_1$ by the abstract closed loop $C_2\times S_2$, i.e., for every solution
$(0,v_1,(x_{c,1},x_{s,1}), y_1)$ of $C_1 \times S_1$
on $\mathbb{Z}_{+}$ there exists a solution
$(0,v_2,(x_{c,2},x_{s,2}), y_2)$ of $C_2 \times S_2$
on $\mathbb{Z}_{+}$ satisfying
\begin{equation}
\label{e:s:pitfalls:reproducibility}
( x_{s,1}(t), x_{s,2}(t) ) \in Q
\text{\ for all $t \in \mathbb{Z}_{+}$}.
\end{equation}
This reproducibility property is then used to
provide evidence that certain properties that the abstract closed loop
$C_2 \times S_2$ satisfies, actually also hold for the concrete closed
loop $C_1 \times S_1$.
\looseness-1

In the first example, we show that~\ref{e:s:pitfalls:reproducibility}
cannot hold if $C_1$ attains state information only through $Q$, i.e.,
if $C_1$ takes the form $C_1' \circ Q$.
In other words,
the refined controller cannot be symbolic but requires full state information.

\begin{example}
\label{ex:Beispiel5_25Apr14}
We consider the systems
$S_1$  and $S_2$ which we graphically illustrate by

\begin{tikzpicture}[>=latex,thick,shorten >=1pt,node distance=1.5 cm]
\small
\node at (0,0) {\normalsize $S_1:$};

\node[state]  at (1,0)  (A)  {$3$};
\node[state]            (B) [right of=A]   {$1$};
\node[state]            (C) [right of=B]   {$2$};

\path[->] (A) edge[above]       node  {$0$} (B);
\path[->] (C) edge[above]       node  {$1$} (B);

\draw[->]  (A) to [my loop] node[right,near start] {$1$}(A);
\draw[->]  (C) to [my loop] node[right,near start] {$0$}(C);
\draw[->]  (B) to [my loop] node  {} (B);

\end{tikzpicture}
\begin{tikzpicture}[>=latex,thick,shorten >=1pt,node distance=1.5 cm]
\small

\node at (0,0) {\normalsize $S_2:$};

\node[state]  at (1,0)  (A)  {$3$};
\node[state]            (B) [right of=A]   {$1$};

\path[->] (A) edge[above]       node  {$0$} (B);
\draw[->] (A) to [my loop] node[right,near start] {$1$}(A);
\draw[->] (B) to [my loop] (B);

\end{tikzpicture} \\
The input and output alphabets of $S_1$ and $S_2$ are given by
$U=\{0,1\}$ and $Y=\{1,2,3\}$, respectively.
The transition functions should be clear from the illustration, e.g.
$F_1(2,1) = \{ 1 \}$ and
$F_1(1,u) = \{ 1 \}$ for any $u \in U$.
It is also easily verified that the relation $Q$ given by
$
Q
=
\{ (1,1), (2,3), (3,3) \}
$
is an alternating simulation relation from $S_1$ to $S_2$.

Let the abstract controller $C_2$ be static with
$X_{c,2} = \{ 0 \}$, $V_{c,2} = Y$, and
$H_{c,2}(0,3) = \{ (0,3) \}$,
i.e., $C_2$ enables exactly the
control letter $0$ at the abstract state $3$.
If the concrete controller $C_1$ is symbolic,
then, at the initial time, the sets of control letters
enabled at the plant states $2$ and $3$ coincide. Indeed, these sets
must only depend on the associated abstract states, and $Q(2) = Q(3)$.
In addition, by the symmetry of the plant $S_1$, we may assume without
loss of generality that the control letter $0$ is enabled at the
initial time, so that there exists a solution
$(0,v_1,(x_{c,1},x_{s,1}), y_1)$ of the closed loop $C_1 \times S_1$
satisfying $x_{s,1}(0) = x_{s,1}(1) = 2$.
Then the condition \ref{e:s:pitfalls:reproducibility} requires 
$x_{s,2}(0) = x_{s,2}(1) = 3$ to hold for some solution
$(0,v_2,(x_{c,2},x_{s,2}), y_2)$ of $C_2 \times S_2$ -- a requirement
that contradicts the dynamics of $C_2 \times S_2$.
This shows that the property of reproducibility cannot be attained
using a symbolic controller for the plant $S_1$.
The crucial point with this example is that the condition
\ref{cond:ASR} cannot be satisfied if
the choice of $u_1$ depends only on the abstract states
associated with the plant state $x_1$, but not directly on $x_1$
itself.
\end{example}

In the next example we show that a static controller $C_2$ for the
abstraction $S_2$ cannot be refined to a static controller $C_1$ for
the concrete system $S_1$.

\begin{example}
\label{ex:Beispiel4_25Apr14}
We consider the systems $S_1$ and $S_2$ with the transition functions illustrated graphically by\\
\begin{center}
\begin{tikzpicture}[>=latex,thick,shorten >=1pt,node distance=1.7 cm]
\small
\node at (0,0) {\normalsize $S_1:$};

\node[state]  at (1,0)  (A)  {$1$};
\node[state]            (B) [right of=A]   {$2$};
\node[state]            (C) [below of=B]   {$4$};

\path[->] (B) edge      node[above]  {$0$} (A);
\path[->] (B) edge      node[right]  {$1$} (C);

\path[->] (A) edge[bend right]       node  {} (B);
\path[->] (C) edge[bend left]       node  {} (B);

\small

\node at (4,0) {\normalsize $S_2:$};

\node[state]  at (5,0)  (A)  {$1$};
\node[state]            (B) [right of=A]   {$2$};
\node[state]            (C) [below of=B]   {$4$};
\node[state]            (D) [below of=A] {$3$};

\path[->] (B) edge     node[above]  {$0$} (A);
\path[->] (B) edge     node[right]  {$1$} (C);

\path[->] (A) edge[bend left]       node  {} (D);
\path[->] (C) edge[bend left]       node  {} (B);

\path[->] (D) edge node[left ] {$0$} (A);
\path[->] (D) edge node[below] {$1$} (C);

\end{tikzpicture} 
\end{center}
The input alphabet and the output alphabet is given by
$U = \{ 0, 1 \}$ and $Y = \{ 1, 2, 3, 4 \}$, respectively.
It is easily verified that the relation $Q$ given by
$
Q
=
\{ (1,1), (2,2), (2,3), (4,4) \}
$
is an alternating simulation relation from $S_1$ to
$S_2$. In addition, in this example the relation $Q$ satisfies even the
more restrictive requirement that $u_1 = u_2$ holds in~\ref{cond:ASR}.

Suppose that the abstract controller $C_2$ is static and enables
exactly the control letters $0$ and $1$ at the abstract states $2$ and
$3$, respectively. If the concrete controller $C_1$ is static,
then the set of control letters enabled at
the plant state $2$ does not vary with time.
By the symmetry of the
plant $S_1$, we may again assume without loss of generality that the
control letter $0$ is enabled at the state $2$, so that there exists a
solution $(0,v_1,(x_{c,1},x_{s,1}), y_1)$ of the closed loop $C_1 \times S_1$
satisfying $x_{s,1}(0) = x_{s,1}(2) = 1$. Then the condition
\ref{e:s:pitfalls:reproducibility} asks for 
$x_{s,2}(0) = x_{s,2}(2) = 1$ for some solution
$(0,v_2,(x_{c,2},x_{s,2}), y_2)$ of $C_2 \times S_2$ -- a requirement
that contradicts the dynamics of $C_2 \times S_2$.
This shows that the property of reproducibility cannot be attained
using a static controller for the plant $S_1$ despite the fact that
the abstract controller is static.
The crucial point with this example is that the condition
\ref{e:SR_ASR_condition} only mandates that for each transition from
$x_1$ to $x_1'$ in $S_1$ there exists a state $x_2' \in Q(x_1')$ that
is a successor of $x_2$ in $S_2$, but it is not required that every
$x_2' \in Q(x_1')$ succeeds $x_2$; consider e.g.~the case
$x_1 = x_2 = 1$, $x_1' = x_2' = 2$.
As a result, the state $1$ and $4$ cannot precede the
state $2$ and $3$, respectively, in $S_2$, and so, implicitly, the
static controller $C_2$ has some access to the history
of the solution. In contrast, at the state $2$ the dynamics of $S_1$
does not encode analogous information, which in fact could here only
be provided by a controller for $S_1$ that is dynamic rather than
static.
\looseness-1
\end{example}

As our examples show, alternating simulation
relations are not adequate for the controller refinement, whenever
{\bf i)} the concrete controller has merely symbolic state information
and  {\bf ii)} the complexity of the refined controller should not
exceed the complexity of the abstract controller.
Moreover, we
point out that 
in both examples the respective relation $Q$ is not merely an
alternating simulation relation
according to our definition in~\ref{cond:ASR}, but also an
$1$-approximate bisimulation relation and
$1$-approximate alternating bisimulation relation
according to Definitions~9.5 and~9.8 in~\cite{Tabuada09},
respectively. Hence, the latter
concepts also suffer from both issues described in this section.

\section{Feedback Refinement Relations}
\label{s:FeedbackRefinementRelations}

In this section, we introduce \begriff{feedback refinement relations} as
a novel means to compare systems in the context of controller
synthesis, in which we focus on simple systems.

\subsection{Definition and basic properties}

We start by introducing the behavior of a system, where
we follow the notion of \begriff{infinitary completed trace semantics}
\cite{vanGlabbeek01}.

\begin{definition}
\label{def:SystemBehavior}
Let $S$ denote the system \ref{e:def:System}. The set $\mathcal{B}(S)$,
\begin{multline}
\label{e:def:SystemBehavior:ExternalBehavior}
\hspace{-1em}
\mathcal{B}(S)
=
\{(u,y)|
\exists_{v, x, T}
(u,v,x,y) \text{ is a solution of $S$ on $\intco{0;T}$, }
\\
\text{and if $T < \infty$, then } F( x(T-1), v(T-1) ) = \emptyset\},
\end{multline}
is called the \begriff{behavior of $S$}.
\end{definition}
Note that it often occurs that a system is non-continuable for a
certain state-input pair, e.g.~the terminating state of a terminating
program. With our notion of system
behavior, which possibly consists of finite signals as well as infinite signals,
such signals are naturally included as valid elements.
In our definition of system relation below, we need a notion of state
dependent admissible inputs.
For any simple system $S$ of the form \ref{e:def:System}, we define the
set $U_S(x)$ of \begriff{admissible inputs at the state $x\in X$} by
\[
U_S(x)
=
\Menge{u \in U}{F(x,u) \not= \emptyset},
\]
and the image of a subset $\Omega \subseteq X$ under $U_S$ is denoted
$U_S(\Omega)$.

\begin{definition}
\label{def:FeedbackRefinementRelation}
Let $S_1$ and $S_2$ be
simple systems,
\begin{equation}
\label{e:systemMooreId}
S_i
=
(X_i, X_i, U_i, U_i, X_i, F_i, \id)
\end{equation}
for $i \in \{ 1, 2 \}$, 
and assume that $U_2 \subseteq U_1$.
A strict relation $Q \subseteq X_1 \times X_2$ is a
\begriff{feedback refinement relation} from $S_1$ to $S_2$ if the
following holds for all $(x_1,x_2) \in Q$:
\begin{enumerate}
\item
\label{def:FeedbackRefinementRelation:in}
$U_{S_2}(x_2) \subseteq U_{S_1}(x_1)$;
\item
\label{def:FeedbackRefinementRelation:dyn}
$u \in U_{S_2}(x_2)$ \implies
$Q(F_1(x_1,u)) \subseteq F_2(x_2, u)$.
\popQED
\end{enumerate}
\end{definition}

The fact that $Q$ is a feedback refinement relation from $S_1$ to
$S_2$ will be denoted $S_1 \preccurlyeq_Q S_2$, and we write
$S_1 \preccurlyeq S_2$ if $S_1 \preccurlyeq_Q S_2$ holds for some $Q$.

Intuitively, and similarly to simulation relations and their variants,
a feedback refinement relation from a system $S_1$ to a system
$S_2$ associates states of $S_1$ with states of
$S_2$, and imposes certain conditions on the local dynamics of the
systems in the associated states.
However, while e.g.~alternating simulation relations
only require that for each input $u_2$ admissible for $S_2$ there
exists an associated input $u_1$ admissible for $S_1$
\cite{Tabuada09}, our definition above additionally mandates that
$u_1 = u_2$.
Moreover, the definition of (approximate) alternating simulation
relation requires that for each transition from $x_1$ to $x_1'$ in
$S_1$ \emph{there exists}
a state $x_2'$ associated with $x_1'$ and a transition from
$x_2$ to $x_2'$ in $S_2$; see condition \ref{e:SR_ASR_condition}.
In contrast, feedback refinement relations
require the existence of the latter transition for \emph{every} state
$x_2'$ associated with $x_1'$.

We next show that the relation
$\preccurlyeq$ is reflexive and transitive.

\begin{proposition}
\label{prop:FeedbackRefinementRelation:follow}
Let $S_1$, $S_2$ and $S_3$ be simple systems. Then:
\begin{asparaenum}[(a)]
\item
\label{prop:FeedbackRefinementRelation:reflexive}
$S_1 \preccurlyeq_{\id} S_1$.
\item
\label{prop:FeedbackRefinementRelation:transitiv}
If $S_1 \preccurlyeq_Q S_2$ and $S_2 \preccurlyeq_R S_3$, then
$S_1 \preccurlyeq_{R \circ Q} S_3$.
\popQED
\end{asparaenum}
\end{proposition}

\begin{proof}
Suppose that $S_i$ is of the form \ref{e:systemMooreId},
$i \in \{1,2,3\}$.
The requirements in Def.~\ref{def:FeedbackRefinementRelation}
are satisfied with $Q = \id$, $S_1 = S_2$ and $x_1 = x_2$, which
proves \ref{prop:FeedbackRefinementRelation:reflexive}.
To prove \ref{prop:FeedbackRefinementRelation:transitiv},
assume that $S_1 \preccurlyeq_Q S_2 \preccurlyeq_R S_3$. Then
$R \circ Q$ is strict since both $R$ and $Q$ are so,
and $U_3 \subseteq U_1$.
Let $(x_1,x_3) \in R \circ Q$. Then there exists
$x_2 \in X_2$ satisfying $(x_1,x_2) \in Q$ and $(x_2,x_3) \in R$.
Thus,
$U_{S_3}(x_3) \subseteq U_{S_2}(x_2) \subseteq U_{S_1}(x_1)$, and so
the condition \ref{def:FeedbackRefinementRelation:in} in
Def.~\ref{def:FeedbackRefinementRelation} is satisfied with
$R \circ Q$ and $S_3$ in place of $Q$ and $S_2$, respectively.
As for the condition \ref{def:FeedbackRefinementRelation:dyn},
additionally assume that $u \in U_{S_3}(x_3)$. Then
$u \in U_{S_2}(x_2)$, and
$S_1 \preccurlyeq_Q S_2 \preccurlyeq_R S_3$ implies
$Q( F_1( x_1,u) ) \subseteq F_2(x_2,u)$ and
$R( F_2( x_2,u) ) \subseteq F_3(x_3,u)$.
Then
$R( Q( F_1( x_1,u) ) )
\subseteq
F_3(x_3,u)$, and so
$S_1 \preccurlyeq_{R \circ Q} S_3$.
\end{proof}

\subsection{Feedback composability and behavioral inclusion}
\label{ss:FeedbackComposabilityAndBehavioralInclusion}

In the following, we present the main result of this section.
We consider three systems $S_1$, $S_2$ and $C$ and
assume that $C$ is feedback composable with $S_2$.
We first prove that, given a feedback refinement relation $Q$ from
$S_1$ to $S_2$, $Q \circ S_1$ and $S_1$ are, respectively, feedback
composable with $C$ and $C \circ Q$.
Subsequently, we show that the behavior of the closed
loops $C\times (Q\circ S_1)$ and $(C\circ Q)\times S_1$ are both
reproducible by the closed loop $C\times S_2$.

Even though we do not assign any particular role to the systems
$S_1$, $S_2$ and $C$, in foresight of the next section, where we use
our result to develop abstraction-based solutions of general control
problems, we might regard $S_1$, $S_2$ and $C$ as the plant, the
abstraction and controller for the abstraction, respectively. In this
context, we might assume that the state of $S_1$ is accessible only
through the measurement map $Q$. In that case, $Q\circ S_1$ actually
represents the system for which we seek a controller and the behavior
of $\mathcal{B}(C\times (Q\circ S_1))$ is of interest. Alternatively,
we may start with the premise that a controller for $S_1$ needs to be
realizable on a digital device and hence, can accept only a finite
input alphabet. In that case, we may interpret $Q$ as an input
quantizer for the discrete controller $C$ and the behavior of
$\mathcal{B}((C\circ Q)\times S_1)$ is of interest. In any case,
we show that both behaviors are reproduced by the
abstract closed loop $C \times S_2$.
In the rest of the paper, we
identify $\{ 0 \} \times \left( U \times Y \right)$ with
$U \times Y$ in the obvious way.

\begin{theorem}
\label{th:BehavioralInclusionSuff}
Let $Q$ be a feedback refinement relation from the system $S_1$ to the
system $S_2$, and assume
that the system $C$ is feedback
composable with $S_2$.
Then the following
holds.
\begin{enumerate}
\item
\label{th:BehavioralInclusionSuff:item:feddback-composable}
$C$ is feedback composable with $Q \circ S_1$,
and $C \circ Q$ is feedback composable with
$S_1$.
\item
\label{th:BehavioralInclusionSuff:item:QcircS1xC}
$\mathcal{B}(C \times (Q \circ S_1))
\subseteq
\mathcal{B}(C \times S_2)$.
\item
\label{th:BehavioralInclusionSuff:item:S1xCcircQ}
For every
$(u,x_1) \in \mathcal{B}((C \circ Q) \times S_1)$
there exists a map
$x_2$
such that $(u,x_2) \in \mathcal{B}(C \times S_2)$ and
$(x_1(t), x_2(t)) \in Q$ for all
$t$ in the domain of $x_1$.
\popQED
\end{enumerate}
\end{theorem}

\begin{proof}
By our hypotheses,
$S_1$ and $S_2$ are simple,
so we assume that these systems are of the form
\ref{e:systemMooreId}.
Moreover,
\begin{equation}
\label{e:th:BehavioralInclusionSuff:proof:QcircS1}
Q \circ S_1 = ( X_1, X_1, U_1, U_1, X_2, F_1, H_1' ),
\end{equation}
where $U_2 \subseteq U_1$ and
$H_1'$ takes the form $H_1'(x,u) = Q(x) \times \{ u \}$.
Let the system $C$ be of the form
\begin{equation}
\label{e:th:TransferOfClosedLoopProperties:solutions:C}
C = ( X_c, X_{c,0}, U_c, V_c, Y_c, F_c, H_c ),
\end{equation}
and observe that $Y_c \subseteq U_1$ and $X_2 \subseteq U_c$ as
$C$ is feedback composable (f.c.) with $S_2$.
Moreover, since
$X_1 \not= \emptyset$ and $Q$ is strict,
the serial composition $C \circ Q$ is well-defined,
\[
C \circ Q
=
(X_c, X_{c,0}, X_1, V_c, Y_c, F_c, H_c' ),
\]
where $H_c'$ takes the form
$H_c'(x_c,x_1) = H_c(x_c,Q(x_1))$.

To prove \ref{th:BehavioralInclusionSuff:item:feddback-composable}, we
first observe that the conditions
\begin{equation}
\label{e:th:BehavioralInclusion:proof:composable}
x_2 \in Q(x_1),
(u,v) \in H_c(x_c,x_2),
F_1(x_1,u) = \emptyset
\end{equation}
together imply $F_c(x_c, v) = \emptyset$. Indeed, it follows from
\ref{e:th:BehavioralInclusion:proof:composable} and the requirement
\ref{def:FeedbackRefinementRelation:in} in Definition
\ref{def:FeedbackRefinementRelation} that $F_2(x_2,u) = \emptyset$,
and our claim follows as $C$ is f.c.~with $S_2$. This
shows that $C$ is f.c.~with $Q \circ S_1$.
Similarly, let $x_1 \in X_1$, $(u,v) \in H_c'(x_c,x_1)$ and
$F_1(x_1,u) = \emptyset$. Then, by the
definition of $H_c'$, there exists
$x_2 \in Q(x_1)$ such that $(u,v) \in H_c(x_c,x_2)$. Then
\ref{e:th:BehavioralInclusion:proof:composable} holds, and so
$F_c(x_c,v) = \emptyset$ as we have already shown.
Hence, $C \circ Q$ is f.c.~with $S_1$, which completes
the proof of \ref{th:BehavioralInclusionSuff:item:feddback-composable}.

To prove \ref{th:BehavioralInclusionSuff:item:QcircS1xC},
let $(u,x_2) \in \mathcal{B}(C \times (Q \circ S_1))$ be defined on
$\intco{0;T}$, $T \in \mathbb{N} \cup \{ \infty \}$. Then there exist
maps $x_c$, $x_1$ and $v$ such that $(0,(v,u),(x_c,x_1),(u,x_2))$ is a
solution of $C \times (Q \circ S_1)$ on $\intco{0;T}$. Moreover, if
additionally $T < \infty$, then
we also have
\begin{equation}
\label{e:th:BehavioralInclusion:proof:blocking2}
F_c( x_c(T-1), v(T-1) ) = \emptyset
\vee
F_1(x_1(T-1),u(T-1)) = \emptyset.
\end{equation}
By Proposition \ref{prop:ClosedLoop:0},
$(u,u,x_1,x_2)$ is a solution of $Q \circ S_1$ on $\intco{0;T}$, and
$(x_2,v,x_c,u)$ is a solution of $C$ on $\intco{0;T}$. The former fact
implies the following:
\begin{align}
\label{e:th:BehavioralInclusion:proof:b}
\forall_{t \in \intco{0;T}}\;
& x_2(t) \in Q( x_1(t) ),\\
\label{e:th:BehavioralInclusion:proof:c}
\forall_{t \in \intco{0;T-1}}\;
& x_1(t+1) \in F_1(x_1(t), u(t)).
\end{align}
We claim that $(u,u,x_2,x_2)$ is a solution of $S_2$, so that
$(0,(v,u),(x_c,x_2),(u,x_2))$ is a solution of $C \times S_2$ by
Proposition \ref{prop:ClosedLoop:0}.
First, we observe that
$F_2(x_2(t),u(t)) \not= \emptyset$ for every $t \in \intco{0;T-1}$.
Indeed, $(u(t),v(t)) \in H_c(x_c(t),x_2(t))$ for every such $t$ since
$(x_2,v,x_c,u)$ is a solution of $C$ on $\intco{0;T}$.
Hence,
$F_2(x_2(t),u(t)) = \emptyset$ for some $t \in \intco{0;T-1}$
implies $F_c(x_c(t),v(t)) = \emptyset$ as $C$ is f.c.~with $S_2$.
This is a contradiction as $x_c(t+1) \in F_c(x_c(t),v(t))$, so
$F_2(x_2(t),u(t)) \not= \emptyset$ for every $t \in \intco{0;T-1}$.
Consequently, $u(t) \in U_{S_2}(x_2(t))$ for all $t \in \intco{0;T-1}$, so
\ref{e:th:BehavioralInclusion:proof:b},
\ref{e:th:BehavioralInclusion:proof:c} and the requirement
\ref{def:FeedbackRefinementRelation:dyn} in Definition
\ref{def:FeedbackRefinementRelation} imply that
$x_2(t+1) \in F_2(x_2(t),u(t))$ for all $t \in \intco{0;T-1}$. This
shows that $(0,(v,u),(x_c,x_2),(u,x_2))$ is a solution of $C \times S_2$
on $\intco{0;T}$.

Finally, we see that if $T < \infty$ and $u(T-1) \in U_{S_2}(x_2(T-1))$,
then \ref{e:th:BehavioralInclusion:proof:b} and the requirement
\ref{def:FeedbackRefinementRelation:in} in Definition
\ref{def:FeedbackRefinementRelation} together imply
$F_1(x_1(T-1),u(T-1)) \not= \emptyset$, and in turn,
\ref{e:th:BehavioralInclusion:proof:blocking2} shows that
$F_c(x_c(T-1),v(T-1)) = \emptyset$.
Thus, $(u,x_2) \in \mathcal{B}(C \times S_2)$, which proves
\ref{th:BehavioralInclusionSuff:item:QcircS1xC}.

To prove \ref{th:BehavioralInclusionSuff:item:S1xCcircQ},
let $(u,x_1) \in \mathcal{B}((C \circ Q) \times S_1)$ be
defined on $\intco{0;T}$, $T \in \mathbb{N} \cup \{ \infty \}$. Then
there exist maps $x_c$ and $v$ such that $(0,(v,u),(x_c,x_1),(u,x_1))$
is a solution of $(C \circ Q) \times S_1$ on $\intco{0;T}$. Moreover,
if additionally $T < \infty$, then
we also have
\begin{equation}
\label{e:th:BehavioralInclusion:proof:blocking1}
F_c(x_c(T-1), v(T-1)) = \emptyset
\vee
F_1(x_1(T-1),u(T-1)) = \emptyset.
\end{equation}
By Proposition \ref{prop:ClosedLoop:0},
$(u,u,x_1,x_1)$ and $(x_1,v,x_c,u)$ is a solution of $S_1$ and
$C \circ Q$, respectively. In particular, by the definition of $H_c'$,
there exists a map $x_2 \colon \intco{0;T} \to X_2$ such that
$x_2(t) \in Q(x_1(t))$ and $(u(t),v(t)) \in H_c(x_c(t),x_2(t))$ for
all $t \in \intco{0;T}$.
Then $(x_2,v,x_c,u)$ and $(u,u,x_1,x_2)$ is a solution of $C$ and
$Q \circ S_1$, respectively, so
$(0,(v,u),(x_c,x_1),(u,x_2))$ is a solution of $C \times (Q \circ S_1)$ by
Proposition \ref{prop:ClosedLoop:0}.
We next observe that if $T < \infty$ and
$F_1(x_1(T-1),u(T-1)) \not= \emptyset$, then
\ref{e:th:BehavioralInclusion:proof:blocking1} implies
$F_c(x_c(T-1),v(T-1)) = \emptyset$. This shows that
$(u,x_2) \in \mathcal{B}(C \times (Q \circ S_1))$, and so
\ref{th:BehavioralInclusionSuff:item:S1xCcircQ} follows from
\ref{th:BehavioralInclusionSuff:item:QcircS1xC}.
\end{proof}

Next we show, that feedback
refinement relations are not only sufficient, but indeed necessary
for the controller refinement as considered in this paper.

\begin{theorem}
\label{th:BehavioralInclusionNecess}
Let $S_1$ and $S_2$ be simple systems of the
form~\ref{e:systemMooreId}, and let $Q \subseteq X_1\times X_2$ be a
strict relation.
If for every system $C$ that
is feedback composable with $S_2$ follows that 
$C$ is feedback composable with $Q\circ S_1$ and
$\mathcal{B}(C \times (Q \circ S_1))
\subseteq
\mathcal{B}(C \times S_2)$
holds,
then $Q$ is a feedback refinement relation from $S_1$ to $S_2$.
\end{theorem} 
\begin{proof}\label{pageref3:review:item31}
In the proof we consider systems $Q \circ S_1$ of the
form~\ref{e:th:BehavioralInclusionSuff:proof:QcircS1}.
Let $C$ be given by
$(\{0\},\{0\},X_2,\{0\},U_2,F_c,H_c)$ with $F_c(0,0) = \emptyset$ and $H_c$
being strict.
Then, $C$ is feedback composable (f.c.) with
$S_2$, and in turn, $C$ is f.c.~with $Q \circ S_1$ by
our hypothesis.
This implies
$U_2 \subseteq U_1$ as required in
Def.~\ref{def:FeedbackRefinementRelation}.

To prove that $Q$ satisfies the condition
\ref{def:FeedbackRefinementRelation:in} in
Def.~\ref{def:FeedbackRefinementRelation}, we let $(x_1,x_2) \in Q$
and $u \in U_{S_2}(x_2)$ and
show that 
$F_1(x_1,u) \not= \emptyset$. Let $C$ be given by 
$(\{0\},\{0\},X_2,X_2,U_2,F_c,H_c)$ with 
$H_c(0,x'_2) = \{ (u,x'_2) \}$ for all $x'_2 \in X_2$ and 
$F_c(0,x_2) = \{0\}$ and 
$F_c(0,x'_2) = \emptyset$ for 
$x'_2 \in X_2 \setminus \{ x_2 \}$. Then $C$ is f.c.~with $S_2$. In particular, the condition
\ref{e:def:ClosedLoop:composable} in Definition \ref{def:ClosedLoop}
reduces to $F_2(x_2,u) \not= \emptyset$. Then $C$ is also f.c.~with $Q \circ S_1$ by our hypothesis, and here the condition
\ref{e:def:ClosedLoop:composable} implies
$F_1(x_1,u) \not= \emptyset$ and the claim follows.

To prove that $Q$ satisfies the condition
\ref{def:FeedbackRefinementRelation:dyn} in Definition~\ref{def:FeedbackRefinementRelation},
we choose $C$ by $(\{0\},\{0\},X_2,X_2,U_2,F_c, H_c)$  with $H_c$ and $F_c$
defined by: if $U_{S_2}(x_2)=\emptyset$ we set $H_c(0,x_2)=U_2 \times \{x_2\}$ and  
$F_c(0,x_2)= \emptyset$; otherwise $H_c(0,x_2)=U_{S_2}(x_2) \times \{x_2\}$ and $F_c(0,x_2)=\{0\}$.
With this definition of $C$ 
condition~\ref{e:def:ClosedLoop:composable}  holds and $C$ is
f.c.~with $S_2$, and by our hypothesis, $C$ is also f.c.~with $Q\circ S_1$.
Suppose that 
condition \ref{def:FeedbackRefinementRelation:dyn} does not hold,
then
there exist
$(x_1,x_2) \in Q$, $u \in U_{S_2}(x_2)$, $x'_1\in F_1(x_1,u)$ and
$x_2' \in Q(x_1')$ such that $x_2' \not \in F_2(x_2,u)$. 
Let $\bar x_1=x_1x_1'$ and $\bar u=uu'$ with $(u',x_2') \in H_c(0,x_2')$. Then
$(\bar u,\bar u,\bar x_1,\bar x_1)$ is a solution of $S_1$ on 
$\intco{0;2}$.
Define $\bar x_2=x_2x_2'$ and observe that 
$(\bar u,\bar u,\bar x_1,\bar x_2)$ is a solution of $Q\circ S_1$.
Let $\bar x_c=00$, since $F_2(x_2,u)\neq \emptyset$, we see that
$(u,x_2) \in H_c(0,x_2)$ and $\{0\}=F_c(0,x_2)$. Also
$(u',x_2') \in H_c(0,x_2')$ by 
our choice of $u'$
and thus $(\bar x_2,\bar x_2,\bar x_c,\bar u)$ is a
solution of $C$. Hence by
Proposition~\ref{prop:ClosedLoop:0} we see that 
$(0,(\bar x_2,\bar u),(\bar x_c,\bar x_1),(\bar u,\bar x_2))$ is a solution of $C\times (Q\circ S_1)$.
Consider $(\hat u,\hat x_2)\in \mathcal{B}(C\times (Q\circ
S_1))$ with $\hat u|_{\intco{0;2}}=\bar u$ and $\hat
x_2|_{\intco{0;2}}=\bar x_2$.
Since $\bar x_2(1)\notin F_2(\bar x_2(0),\bar u(0))$ the sequence
$(0,(\bar x_2, \bar u),(\bar x_c,\bar x_2),(\bar u,\bar x_2))$ cannot be a solution of $C\times S_2$,
and so $(\hat u,\hat x_2) \notin \mathcal{B}(C \times S_2)$.
This is a contradiction, which establishes condition
\ref{def:FeedbackRefinementRelation:dyn} in Definition~\ref{def:FeedbackRefinementRelation}.
\end{proof}

\section{Symbolic Controller Synthesis}
\label{s:SymbolicSynthesis}

In this section, we propose a controller synthesis technique based on
the concept of feedback refinement relations which resolves the state
information and refinement complexity issues as explained and
illustrated in Sections \ref{s:intro} and \ref{s:pitfalls}, applies
to general specifications, and produces controllers that are robust
with respect to various disturbances.
We follow the general three step procedure of
abstraction-based synthesis outlined in Section \ref{s:intro}, where
we focus on the first and third steps. Our results will be
complemented by the computational method presented in Section
\ref{s:Construction}, whereas the solution of the abstract control
problem -- the second step of the general procedure -- is beyond the
scope of the present paper. Indeed, large classes of these problems
can be solved efficiently using standard algorithms,
e.g.~\cite{EmersonClarke82,PnueliRosner89,Vardi95,Tabuada09,BloemJobstmanPitermanPnueliSaar12,i13absocc}.
\looseness-1

\subsection{Solution of control problems}
\label{ss:SolutionOfControlProblems}

We begin with the definition of the synthesis problem.

\begin{definition}
\label{d:specification}
Let $S$ denote the system \ref{e:def:System}.
Given a set $Z$, any subset $\Sigma \subseteq Z^{\infty}$ is called a
\begriff{specification on $Z$}.
A system $S$ is said to \begriff{satisfy} a specification $\Sigma$
on $U \times Y$ if $\mathcal{B}(S) \subseteq \Sigma$.
Given a specification $\Sigma$ on $U \times Y$, the system
$C$ \begriff{solves the control problem $(S,\Sigma)$} if $C$ is
feedback composable with $S$ and the closed loop $C \times S$
satisfies $\Sigma$.
\end{definition}

It is clear that we can use linear temporal logic (LTL) to define a specification for a given
system $S$. Indeed, suppose that we are given a finite set $\mathcal{P}$ of
atomic propositions, a labeling function
$L \colon U\times Y\rightrightarrows \mathcal{P}$ and an LTL formula $\varphi$ defined
over $\mathcal{P}$, see
e.g.~\cite[Chapter~5]{BaierKatoen08}. Then we can formulate the
control problem $(S,\Sigma)$ to enforce the formula $\varphi$ on $S$
using the specification
\begin{IEEEeqnarray*}{c}
\Sigma=\{(u,y)\in (U\times Y)^{\mathbb{Z}_{+}}\mid L\circ (u,y)\text{
  satisfies } \varphi\}.
\end{IEEEeqnarray*}
Our notion of specification is not limited to LTL,
e.g. ``$y(t)=1$ holds for all even $t\in \mathbb{Z}_+$'' is not expressible in
LTL~\cite[Remark 5.43]{BaierKatoen08}, but is a valid specification in
our framework.

We are now going to solve control problems using
Theorem \ref{th:BehavioralInclusionSuff}. 
As we have already discussed, the concrete control problem
$(S_1,\Sigma_1)$ will not be solved directly. Instead, we will
consider an auxiliary problem for the abstraction (``abstract control
problem''), whose solution will induce a solution of the concrete problem.

\begin{definition}
\label{def:AbstractSpecification}
Let $S_1$ and $S_2$ be simple systems of the
form~\ref{e:systemMooreId},
let $\Sigma_1$ be a specification on $U_1 \times X_1$, and let
$Q \subseteq X_1 \times X_2$ be a strict relation.
A specification $\Sigma_2$ on $U_2 \times X_2$ is called an
\begriff{abstract specification} associated with $S_1$, $S_2$, $Q$ and
$\Sigma_1$, if the following condition holds.\\
If $(u,x_2) \in \Sigma_2$, where $x_2$ and $u$ are defined on
$\intco{0;T}$ for some $T \in \mathbb{N} \cup \{\infty\}$, and if
$x_1 \colon \intco{0;T} \to X_1$ satisfies
$(x_1(t), x_2(t)) \in Q$ for all $t \in \intco{0;T}$, then
$(u,x_1) \in \Sigma_1$.
\end{definition}

For the sake of simplicity, we write
$(S_1,\Sigma_1) \preccurlyeq_Q (S_2,\Sigma_2)$
whenever $S_1 \preccurlyeq_Q S_2$ and $\Sigma_2$ is an abstract
specification associated with $S_1$, $S_2$, $Q$ and $\Sigma_1$.
The result presented below shows how to use a solution of the abstract
control problem to arrive at a solution of the concrete control
problem, resulting in the closed loop in \ref{f:ClosedLoop_FRR}.

\begin{theorem}
\label{th:refinement}
If $(S_1,\Sigma_1) \preccurlyeq_Q (S_2,\Sigma_2)$ and
the abstract controller $C$ solves the
control problem
$(S_2,\Sigma_2)$, then the refined controller $C \circ Q$ solves the
control problem $(S_1,\Sigma_1)$.
\end{theorem}

\begin{proof}
As $C$ solves $(S_2,\Sigma_2)$, $C$ is feedback composable with $S_2$,
and hence, $C \circ Q$ is feedback composable with $S_1$ by Theorem
\ref{th:BehavioralInclusionSuff}.

\noindent
It remains to show that
$\mathcal{B}((C \circ Q) \times S_1) \subseteq \Sigma_1$. So, let
$(u,x_1) \in \mathcal{B}((C \circ Q) \times S_1)$ be arbitrary and
invoke Theorem \ref{th:BehavioralInclusionSuff} again to see that there
exists a map $x_2$ such that
$(u,x_2) \in \mathcal{B}(C \times S_2)$ and
$(x_1(t), x_2(t)) \in Q$ for all $t$ in the domain of $x_2$.
Then $(u,x_2) \in \Sigma_2$ since $C$ solves $(S_2,\Sigma_2)$, and
the definition of the abstract specification $\Sigma_2$ shows that
$(u,x_1) \in \Sigma_1$.
\end{proof}

\subsection{Uncertainties and disturbances}
\label{ss:Uncertainties}

We next show that it is an easy task in our framework
to synthesize controllers that are robust with respect to
various disturbances including
plant uncertainties, input disturbances and
measurement errors. In particular, we demonstrate that the
synthesis of a robust controller can
be reduced to the solution of an auxiliary, unperturbed  control
problem.

Let us consider the closed loop
illustrated in~\ref{fig:PerturbedClosedLoop} consisting of a plant
given by a
simple system $S_1$ of the form \ref{e:systemMooreId},
the \begriff{perturbation maps} $P_i$, given by
strict set-valued  maps with non-empty domains
\begin{IEEEeqnarray}{c}\label{e:PerturbationMaps}
\begin{IEEEeqnarraybox}[][c]{l'l}
P_1:\hat U_1\rightrightarrows U_1,& P_2:X_1\rightrightarrows \hat X_1,\\
P_3:\hat U_1\rightrightarrows Y_1,&P_4:X_1\rightrightarrows Y_2,
\end{IEEEeqnarraybox}
\end{IEEEeqnarray}
and a strict quantizer
\begin{equation}
\label{e:PerturbationQuantizer}
Q \colon \hat X_1\rightrightarrows X_2.
\end{equation}
We seek to synthesize a controller given as a system
\begin{equation}
\label{e:PerturbationController}
C
=
( X_c, X_{c,0}, X_2, V_c, \hat U_1, F_c, H_c ),
\end{equation}
to robustly enforce a given specification
$\Sigma_1$ on $Y_1 \times Y_2$.

The \begriff{behavior} of the
closed loop in~\ref{fig:PerturbedClosedLoop} is defined as the set of
all sequences $(y_1,y_2)\in (Y_1\times Y_2)^{\intco{0;T}}$,
$T \in \mathbb{N} \cup \{ \infty \}$, for which there exist a solution
$(u,u,x,x)$ of $S_1$ on $\intco{0;T}$ and a solution
$(u_c,v_c,x_c,y_c)$ of $C$ on $\intco{0;T}$ that
satisfy the following two conditions:
\begin{enumerate}
\item
\label{ii:BehaviorPerturbed}
For all $t \in \intco{0;T}$ we have
\begin{IEEEeqnarray}{c}\label{eq:ii:BehaviorPerturbed}
\begin{IEEEeqnarraybox}[][c]{l'l}
u(t) \in P_1(y_c(t)),& u_c(t) \in Q(P_2(x(t))), \\
y_1(t) \in P_3(y_c(t)), & y_2(t) \in P_4(x(t)).
\end{IEEEeqnarraybox}
\end{IEEEeqnarray}
\item
\label{iii:BehaviorPerturbed}
If $T < \infty$, then
\begin{IEEEeqnarray}{c}\label{eq:iii:BehaviorPerturbed}
\begin{IEEEeqnarraybox}[][c]{l,s}
F_1(x(T-1),u(T-1)) = \emptyset,& or\\
F_c(x_c(T-1),v_c(T-1)) = \emptyset.
\end{IEEEeqnarraybox}
\end{IEEEeqnarray}
\end{enumerate}

\begin{figure}[ht]
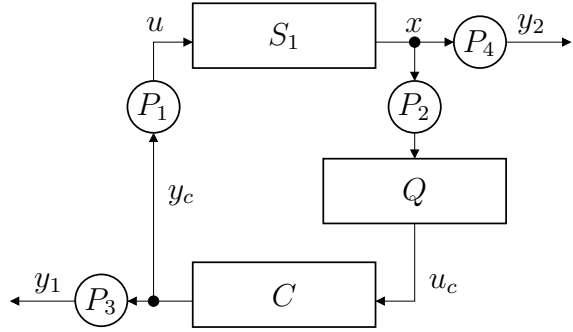

\begin{center}
\psfrag{S}[][]{$S_1$}
\psfrag{Q}[][]{$Q$}
\psfrag{C}[][]{$C$}
\psfrag{P1}[][]{$P_1$}
\psfrag{P2}[][]{$P_2$}
\psfrag{P3}[][]{$P_3$}
\psfrag{P4}[][]{$P_4$}
\psfrag{y1}[][]{$y_1$}
\psfrag{y2}[][]{$y_2$}
\psfrag{xs}[][]{$x$}
\psfrag{us}[r][r]{$u$}
\psfrag{uc}[l][l]{$u_c$}
\psfrag{yc}[l][l]{$y_c$}
\psfrag{z}[l][l]{}
\pgfdeclareimage[height=\ifCLASSOPTIONonecolumn.24\else.43\fi\linewidth]{PerturbedClosedLoop}{figures/ClosedLoop_Perturbed}%
\noindent
\pgfuseimage{PerturbedClosedLoop}
\end{center}
\caption{\label{fig:PerturbedClosedLoop}
Various perturbations in the closed loop.
}
\end{figure}
It is straightforward to
observe, that the perturbations maps $P_1$ and $P_2$ may be used to
model input disturbances and measurement errors, respectively.
We assume that the uncertainties of the dynamics of $S_1$ have
already been modeled by the set-valued transition function $F_1$.
The controller $C$ and the quantizer $Q$, which will usually be discrete,
are not subject to any additional perturbations either.
The maps
$P_3$ and $P_4$ are useful in the presence of output disturbances.
For example, the plant $S_1$ might
represent a sampled variant of a continuous-time control system and
the specification of the desired behavior is naturally formulated in
continuous time, rather than in discrete time. In that context, one can use $P_3$ and $P_4$ to
``robustify'' the specification like 
in~\cite{FainekosGirardKressGazitPappas09} such that properties of the
sampled behavior carry over to the continuous-time behavior.
\looseness-1

Given some specifications $\Sigma_1$ on $Y_1\times Y_2$ and $\hat \Sigma_1$ on
\mbox{$\hat U_1\times X_1$}, we call $\hat \Sigma_1$ a \begriff{robust specification of
$\Sigma_1$ w.r.t. $P_3$ and $P_4$} if for the functions $(y_c,x,y_1,y_2)\in
(\hat U_1\times X_1\times Y_1\times Y_2)^{\intco{0;T}}$, $T \in \mathbb{N} \cup \{\infty\}$,
we have that
\begin{IEEEeqnarray*}{c,t,c,c,c}
(y_c,x)\in \hat\Sigma_1& and &
\forall_{t\in \intco{0;T}}&y_1(t)\in P_3(y_c(t)), & y_2(t)\in P_4(x(t))
\end{IEEEeqnarray*}
implies $(y_1,y_2)\in \Sigma_1$.

In the following result, we present sufficient conditions for a controller $C$ to
robustly enforce a given specification $\Sigma_1$ on the
perturbed closed loop illustrated in \ref{fig:PerturbedClosedLoop}, in
terms of the auxiliary
simple system $\hat S_1$,
\begin{IEEEeqnarray}{l}\label{e:pert:sys:aux}
\begin{IEEEeqnarraybox}[][c]{l}
\hat S_1 = ( X_1, X_1, \hat U_1, \hat U_1, X_1, \hat F_1, \id ),\\
\hat F_1( x, u ) = F_1( x, P_1(u) ),
\end{IEEEeqnarraybox}
\end{IEEEeqnarray}
together with a robust specification $\hat \Sigma_1$ of $\Sigma_1$.
We show in the subsequent corollary, which follows immediately by Theorem~\ref{th:refinement}, how to use an abstraction
$(S_2,\Sigma_2)$ to synthesize such a controller $C$.

\begin{theorem}
\label{th:PerturbedRefinement}
Consider a simple system $S_1$,
perturbation maps $P_i$, $i\in\intcc{1;4}$,
a strict quantizer $Q$, and a controller $C$
as illustrated in~\ref{fig:PerturbedClosedLoop} and respectively
defined in~\ref{e:systemMooreId}, \ref{e:PerturbationMaps},
\ref{e:PerturbationQuantizer} and \ref{e:PerturbationController},
and assume that $F_1$ is strict.
Let $\Sigma_1$ be a specification on $Y_1\times Y_2$. Let $(\hat
S_1,\hat \Sigma_1)$ be an auxiliary control problem,
 where $\hat S_1$ follows from
$S_1$ according to~\ref{e:pert:sys:aux} and $\hat \Sigma_1$ is a
robust specification of $\Sigma_1$ w.r.t. $P_3$ and~$P_4$.

If
$C \circ \hat Q$, with $\hat Q = Q \circ P_2$, solves the
control problem $(\hat S_1, \hat \Sigma_1)$, then the behavior
of the perturbed closed loop in \ref{fig:PerturbedClosedLoop} is a
subset of $\Sigma_1$.
\end{theorem}

\begin{proof}
Our assumptions imply that $C \circ \hat Q$ is feedback
composable with $\hat S_1$. Using Definition \ref{def:ClosedLoop},
Proposition \ref{prop:ClosedLoop:0}, the strictness of $F_1$, and the properties
\ref{eq:ii:BehaviorPerturbed}-\ref{eq:iii:BehaviorPerturbed}, it is
straightforward to show that
$(y_1,y_2)$ is an element of the behavior of the closed loop in
\ref{fig:PerturbedClosedLoop} iff there exists
$(y_c,x) \in \mathcal{B}((C \circ \hat Q) \times \hat S_1)$
satisfying $y_1(t) \in P_3(y_c(t))$ and $y_2(t) \in P_4(x(t))$
for all $t$.
Consequently, if $(y_1,y_2)$ is an element of the behavior of the
closed loop in \ref{fig:PerturbedClosedLoop}, then there exist
$(y_c,x) \in \hat \Sigma_1$ satisfying $y_1(t) \in P_3(y_c(t))$
and $y_2(t) \in P_4(x(t))$ for all $t$, and so
$(y_1,y_2) \in \Sigma_1$ by the definition of $\hat \Sigma_1$.
\end{proof}

\begin{corollary}
\label{c:PerturbedRefinement}
In the context of Theorem~\ref{th:PerturbedRefinement}, if $C$ solves
an abstract control problem
$(S_2,\Sigma_2)$ with
$(\hat S_1,\hat \Sigma_1)\preccurlyeq_{\hat Q} (S_2,\Sigma_2)$,
where $X_2$ is the state space of $S_2$,
then the behavior of the closed loop
in~\ref{fig:PerturbedClosedLoop} is a subset of~$\Sigma_1$.
\end{corollary}

In the following example we demonstrate
that it is crucial to
account for the measurement errors $P_2$ in terms of the
auxiliary quantizer $\hat Q=Q \circ P_2$, as opposed to accounting
for those type of disturbances in terms of an
alternative auxiliary system
$\tilde S_1 = ( X_1, X_1, \hat U_1, \hat U_1, X_1, \tilde F_1, \id )$
with $\tilde F_1$ given by
\begin{IEEEeqnarray}{c}\label{e:aux2}
\tilde F_1(x_1,u)=P_2(F_1(x_1,P_1(u))).
\end{IEEEeqnarray}

\begin{example}
\label{ex:Robustness}
We consider the simple system $S_1$ of the form \ref{e:systemMooreId}
with the transition function illustrated graphically
\begin{center}
\begin{tikzpicture}[>=latex,thick,shorten >=1pt,node distance=2 cm]
\small

\node[state]  at (1,0)  (A)  {$a$};
\node[state]  at (3,0)   (B)  {$b$};
\node[state]  at (1,-1.2)  (C)  {$c$};
\node[state] at (3,-1.2) (D)  {$d$};

\path[->] (A) edge     node[above]  {$\{0,1\}$} (B);
\path[->] (A) edge     node[right]  {$\{0,1\}$} (C);

\path[->] (B) edge       node[right]  {$1$} (D);
\path[->] (C) edge       node[above]  {$0$} (D);

\path[->] (B) edge[loop right] node  {$0$} (B);
\path[->] (C) edge[loop left ] node  {$1$} (C);

\path[->] (D) edge[loop right]   node[right]  {$\{0,1\}$} (D);

\end{tikzpicture} 
\end{center}
The state and input alphabet are given by
$X_1=\{a,b,c,d\}$ and $U_1=\{0,1\}$, respectively.
Suppose we
are given the specification $\Sigma_1$ on $U_1\times X_1$ defined
implicitly by $(u,x)\in \Sigma_1$ iff $d$ is in the image
of $x$.
Let us consider the quantizer $Q=\id$ and the perturbation maps
$P_1=P_3=P_4=\id$ and $P_2$ defined by $P_2(a)=\{a\}$, $P_2(b)=P_2(c)=\{b,c\}$ and
$P_2(d)=\{d\}$.
Let the auxiliary system $\tilde S_1$
coincide with $S_1$ except the transition function is given by
$\tilde F_1(x,u)=P_2(F_1(x,u))$. 

The controller $C\circ Q$, with $C$ given as static system with
strict transition function and
output map
$H_c \colon \{0\} \times X_1 \rightrightarrows U_1 \times X_1$ defined
by
$H_c(0,a) = H_c(0,d) = U_1 \times \{ a \}$,
$H_c(0,b) = \{ (1,a) \}$,
$H_c(0,c) = \{ (0,a) \}$
solves the control problem $(\tilde S_1,\Sigma_1)$. 
However, $(u,x) = \big((0,a),(1,c),(1,c),(1,c),\ldots\big)$
is an element
of the behavior of the closed loop
according to~\ref{fig:PerturbedClosedLoop} and yet violates the specification
$\Sigma_1$.
\end{example}
As the example demonstrates, we cannot rely on the auxiliary system
with transition function~\ref{e:aux2} to synthesize a robust
controller but we need a quantizer that is robust with respect to disturbances. That is essentially
expressed by requiring that $C \circ \hat Q$ with $\hat Q=Q\circ
P_2$ solves the auxiliary control problem $(\hat S_1,\hat \Sigma_1)$.
Intuitively, we require that the controller $C$ ``works'' with any
quantizer symbol $x_2\in Q( P_2(x_1))$ no matter how the disturbance
$P_2$ is acting on the state $x_1$.
Note that in
Example~\ref{ex:Robustness}, the controller $C\circ (\id \circ P_2)$
does not solve the control problem $(\hat S_1,\hat \Sigma_1)$ (which
in this case equals $(S_1,\Sigma_1)$).

Finally,
we would like to mention that in the
context of control systems,  \emph{any}
symbolic controller synthesis procedure that is based on a
deterministic quantizer
is bound to be \begriff{non-robust}.
Indeed, consider the context of
Theorem~\ref{th:PerturbedRefinement} and suppose that
$X_1=\mathbb{R}^n$, $X_2$ is a partition of $X_1$ and let
$P_2(x_1)$ equal the closed Euclidean ball with radius
$\varepsilon \ge 0$ centered at $x_1$.
Let us consider the deterministic quantizer
\mbox{$Q=\:\in\:$}. Then $\hat Q=Q\circ P_2$ is deterministic only in the
degenerate case $\varepsilon=0$.

\section{Canonical Feedback Refinement Relations}
\label{ss:FRRsAndAttainableSets}

In this section, we show that the set membership relation $\in$,
together with an abstraction whose
state alphabet is a cover of the concrete state
alphabet is canonical.
A \begriff{cover} of a set $X$ is a set of subsets of $X$ whose union
equals $X$.

We show that 
$(S_1,\Sigma_1)\preccurlyeq_Q
(S_3,\Sigma_3)$ implies that there
exist $(S_2,\Sigma_2)$, with $X_2$ being a cover of $X_1$ by non-empty
subsets, together with a relation $R$ such that the following holds:
\begin{IEEEeqnarray*}{c}
(S_1,\Sigma_1)
\preccurlyeq_\in
(S_2,\Sigma_2)
\preccurlyeq_R
(S_3,\Sigma_3).
\end{IEEEeqnarray*}
This implies that if we can solve the concrete control problem
$(S_1,\Sigma_1)$ using some abstract control problem $(S_3,\Sigma_3)$,
then we can equally use an abstract control problem $(S_2,\Sigma_2)$
with $X_2$ being to a cover of $X_1$ by non-empty subsets.
Moreover, $(S_2,\Sigma_2)$ can be derived from the problem
$(S_3,\Sigma_3)$ and the quantizer $Q$ alone and is otherwise
independent of $(S_1,\Sigma_1)$.

\subsection{Canonical abstractions}
\label{ss:Canonical}

\begin{proposition}
\label{prop:CanonicalFRR}
Let $S_1$ and $S_2$ be simple systems of the form~\ref{e:systemMooreId},
in which $X_2$ is a cover of $X_1$ by
non-empty subsets and $U_2 \subseteq U_1$.
Then $S_1 \preccurlyeq_{\in} S_2$ iff the following conditions hold.
\begin{enumerate}
\item
\label{prop:CanonicalFRR:in}
$x \in \Omega \in X_2$ implies
$U_{S_2}(\Omega) \subseteq U_{S_1}(x)$.
\item
\label{prop:CanonicalFRR:dyn}
If
$\Omega, \Omega' \in X_2$,
$u\in U_{S_2}(\Omega)$ and
$\Omega' \cap F_1(\Omega,u) \not= \emptyset$, then
$\Omega' \in F_2(\Omega,u)$.
\popQED
\end{enumerate}
\end{proposition}

The above result, whose straightforward proof we omit, will be used in
our proof of the canonicity result, Theorem
\ref{th:CanonicalFRRsExist}. It additionally indicates
constructive methods to compute
a canonical abstraction $S_2$ of a plant $S_1$ if the abstract state space
$X_2$ and the input alphabet $U_2 \subseteq U_1$ are
given. From condition \ref{prop:CanonicalFRR:dyn}
it follows that, if $\Omega \in X_2$, $u \in U_2$ and
$F_1(x,u) \not= \emptyset$ for every $x \in \Omega$, then we may
either choose $F_2(\Omega,u)$ to be empty, which is of course not
desirable\footnote{One should always choose $F_2(\Omega,u)\neq \emptyset$, since it
enlarges the set of control letters available to
any abstract controller and thereby facilitates the solution of the
abstract control problem.}, or ensure that the latter
set contains every \begriff{cell} $\Omega'$ that intersects the attainable set
$F_1(\Omega,u)$ of the cell $\Omega$ under the control letter
$u$. This can be achieved by numerically over-approximating attainable
sets, for which many
algorithms are available, see e.g.
\cite{i11abs} and Section~\ref{s:Construction}.

On the other hand, condition \ref{prop:CanonicalFRR:in} requires that 
$F_2(\Omega,u)$ is empty whenever $F_1(x,u)$ is so for some
$x \in \Omega$.
This raises the question of how to detect the phenomenon of blocking
of the dynamics of the plant. If the transition function $F_1$ is
explicitly given, we assume that its description directly facilitates
the detection of blocking. In the case that the plant represents a
sampled system, so that $F_1$ is the time-$\tau$-map of some
continuous-time control system, blocking can usually be detected in the
course of over-approximating attainable sets.
For example, if an
over-approximation $W$ of the attainable set $F_1(\Omega,u)$ is
computed using interval arithmetic, and if $F_1(x,u) = \emptyset$ for
some $x \in \Omega$, then $W$ will be unbounded, e.g. \cite[Chapter
II.3]{Hartman02}, which is easily detected. 

\subsection{Canonicity result}

Before we state and prove the canonicity result, we introduce a 
technical condition that we impose on the feedback refinement relation $Q$ from 
$(S_1,\Sigma_1)$ to $(S_3,\Sigma_3)$, i.e., 
\begin{condC}
\label{e:th:CanonicalFRRs}
if
$\emptyset \not= Q^{-1}(x) = Q^{-1}(\tilde x)$,
$\emptyset \not= Q^{-1}(x') = Q^{-1}(\tilde x')$,
$\tilde x' \in F_3(\tilde x,u)$, and
$u \in U_{S_3}(x)$,
then $x' \in F_3(x,u)$.%
\end{condC}%
\noindent
We point out that condition \ref{e:th:CanonicalFRRs} is not an essential
restriction and it actually holds for a great variety of
abstractions and relations.
For example, it automatically holds if the abstraction $S_3$ is
defined as a quotient system
\cite[Definition~4.17]{Tabuada09}. In that
case, the elements of $X_3$ correspond to the equivalence
classes of an equivalence relation on $X_1$. Therefore, we have
that $Q^{-1}(x)=Q^{-1}(\tilde x)$ implies
$x=\tilde x$ and condition \ref{e:th:CanonicalFRRs} is trivially
satisfied. Similarly, relations that are based on level sets of
simulation functions \mbox{$V:X_1\times X_3\to \mathbb{R}_{+}$}
with $X_1,X_3\subseteq \mathbb{R}^n$,
see e.g.~\cite{GirardPappas07b}, for popular choices of simulation
functions like $V(x_1,x_3)=\sqrt{(x_1-x_3)^\top P(x_1-x_3)}$ with $P$ being
a positive definite matrix,
where $x^\top$ denotes the transpose of $x$,
satisfy~\ref{e:th:CanonicalFRRs}. In this
case, the relation is given by $Q=\{(x_1,x_3)\in X_1\times X_3\mid V(x_1,x_3)\leq
\varepsilon\}$ and again $Q^{-1}(x)=Q^{-1}(\tilde x)$ implies
$x=\tilde x$ and we conclude that~\ref{e:th:CanonicalFRRs} holds.
Lastly, the condition~\ref{e:th:CanonicalFRRs} also holds, for the
case that $Q$ is given and the abstraction $S_3$
is computed using a deterministic algorithm to over-approximate
attainable sets. This is immediate from the following reformulation of
the condition \ref{def:FeedbackRefinementRelation:dyn} in Definition
\ref{def:FeedbackRefinementRelation}:
If $x_2,x_2' \in X_2$,
$u \in U_{S_2}(x_2)$, and
$Q^{-1}(x_2') \cap F_1(Q^{-1}(x_2),u) \not= \emptyset$, then
$x_2' \in F_2(x_2,u)$.

\begin{theorem}
\label{th:CanonicalFRRsExist}
Let $(S_3,\Sigma_3)$ be a control problem, in which $S_3$ is simple
and of the form \ref{e:systemMooreId}.
Let $X_1$ be any set, and assume that
$Q \colon X_1 \rightrightarrows X_3$ satisfies the condition
\ref{e:th:CanonicalFRRs}.\\
Then there exist a simple system $S_2$ of the form \ref{e:systemMooreId}, a relation
$R \subseteq X_2 \times X_3$ and a specification $\Sigma_2$ on
$U_2 \times X_2$ such that the following holds.
\begin{enumerate}
\def\theenumi{$\ast$}
\item
\label{th:CanonicalFRRsExist:itemCondition}
If $(S_1,\Sigma_1) \preccurlyeq_Q (S_3,\Sigma_3)$
and the system $S_1$ has state space $X_1$, then
$(S_1,\Sigma_1) \preccurlyeq_{\in} (S_2,\Sigma_2) \preccurlyeq_R (S_3,\Sigma_3)$
and
$X_2$ is a cover of $X_1$ by non-empty subsets.
\popQED
\end{enumerate}
\end{theorem}

\begin{proof}
We will prove that \ref{th:CanonicalFRRsExist:itemCondition} holds for
the following choices
of $S_2$, $R$ and $\Sigma_2$:
\ifCLASSOPTIONtwocolumn\relax\else\goodbreak\fi
$X_2
=
\Menge{ \Omega }%
{ \emptyset \not= \Omega = Q^{-1}(x) \wedge x \in X_3}
$,
$R(\Omega)
=
\Menge{x \in X_3}{\Omega = Q^{-1}(x)}
$,
$U_2 = U_3$,
$F_2(\Omega,u)
=
R^{-1} ( F_3( R( \Omega ), u ) )
$,
and $(u,\Omega)\in(U_2\times X_2)^\infty$ is an element of $\Sigma_2$
iff there exists $(u,x_3) \in \Sigma_3$ satisfying
$(\Omega(t),x_3(t)) \in R$ for all $t$ in the domain of $u$.
To establish \ref{th:CanonicalFRRsExist:itemCondition}, assume
that $(S_1,\Sigma_1) \preccurlyeq_Q (S_3,\Sigma_3)$. Then $Q$ is
strict, which already proves our claim on $X_2$, and $S_1$ is simple,
and so we can assume that $S_1$ takes the form \ref{e:systemMooreId}.

To prove $S_1 \preccurlyeq_{\in} S_2$, we first notice that the
condition \ref{prop:CanonicalFRR:in} in Proposition
\ref{prop:CanonicalFRR} is
satisfied. Indeed, let
$x_1 \in \Omega \in X_2$ and $u \in U_{S_2}(\Omega)$. By our choice
of $F_2$ and $R$, there exists $x_3$ satisfying
$(x_1,x_3) \in Q$ and $u \in U_{S_3}(x_3)$.
Then $u \in U_{S_1}(x_1)$
by Def.~\ref{def:FeedbackRefinementRelation} applied to
$S_1 \preccurlyeq_Q S_3$.
To establish the condition \ref{prop:CanonicalFRR:dyn} in
Prop.~\ref{prop:CanonicalFRR},
we let
$\Omega, \Omega' \in X_2$ and $u \in U_{S_2}(\Omega)$ and assume that
$\Omega' \cap F_1(\Omega,u) \not= \emptyset$. By the latter fact there
exist $x_1 \in \Omega$ and $x_1' \in \Omega' \cap F_1(x_1,u)$, and
$u \in U_{S_2}(\Omega)$ implies that there exists $x_3$ such that
$\Omega = Q^{-1}(x_3)$ and $u \in U_{S_3}(x_3)$. We pick $x_3'$
satisfying $\Omega' = Q^{-1}(x_3')$.
Then $(x_1,x_3), (x_1', x_3') \in Q$, and so $S_1 \preccurlyeq_Q S_3$
implies $x_3' \in Q(x_1') \subseteq F_3(x_3,u)$. Hence,
$\Omega' \in F_2(\Omega,u)$
by our choice of $F_2$.
This proves $S_1 \preccurlyeq_{\in} S_2$.

To prove $S_2 \preccurlyeq_R S_3$, let
$(\Omega,x_3) \in R$ and $u \in U_{S_3}(x_3)$ and pick any
$x_1 \in \Omega$. Then $(x_1,x_3) \in Q$
by our choice of $R$, and using 
$S_1 \preccurlyeq_Q S_3$ we obtain $u \in U_{S_1}(x_1)$. The latter
fact implies that there exists $x_1' \in F_1(x_1,u)$, and using 
$S_1 \preccurlyeq_Q S_3$ again we see that
$Q(x_1') \subseteq F_3(x_3,u)$.
Since $Q$ is strict we may pick
$x_3' \in Q(x_1')$. Then $R^{-1}(x_3') \not= \emptyset$, and hence,
$u \in U_{S_2}(\Omega)$
by the definition of $F_2$,
which proves the condition
\ref{def:FeedbackRefinementRelation:in} in Definition
\ref{def:FeedbackRefinementRelation}.
To prove the condition \ref{def:FeedbackRefinementRelation:dyn}
in that definition, let
$(\Omega,x_3) \in R$, $u \in U_{S_3}(x_3)$ and
$\Omega' \in F_2(\Omega,u)$. Then $\Omega' \in R^{-1}(F_3(\Omega,u))$,
so there exist $\tilde x_3$ and $\tilde x_3' \in F_3(\tilde x_3,u)$
satisfying
$\Omega = Q^{-1}(\tilde x_3)$ and $\Omega' = Q^{-1}(\tilde x_3')$.
Then condition \ref{e:th:CanonicalFRRs} implies $x_3' \in F_3(x_3,u)$,
and in turn, $R(\Omega') \subseteq F_3(x_3,u)$.

To complete the proof, we notice that, by the definition of
$\Sigma_2$, $\Sigma_3$ is an abstract specification associated with
$S_2$, $S_3$, $R$ and $\Sigma_2$, which shows
$(S_2,\Sigma_2) \preccurlyeq_R (S_3,\Sigma_3)$.
Finally, to prove $(S_1,\Sigma_1) \preccurlyeq_{\in} (S_2,\Sigma_2)$,
let $(u,\Omega) \in \Sigma_1$, assume that $u$ is defined on
$\intco{0;T}$, and let $x_1 \colon \intco{0;T} \to X_1$ satisfy
$x_1(t) \in \Omega(t)$ for all $t \in \intco{0;T}$.
Then, by the definition of $\Sigma_2$, there exists
$(u,x_3) \in \Sigma_3$ such that
$R(\Omega(t)) = \{x_3(t)\}$ for all $t \in \intco{0;T}$. The latter
condition implies $(x_1(t), x_3(t)) \in Q$, and
$(S_1,\Sigma_1) \preccurlyeq_Q (S_3,\Sigma_3)$ implies
$(u,x_1) \in \Sigma_1$.
\end{proof}

\section{Computation of Abstractions\\ for Perturbed Sampled Control
Systems}
\label{s:Construction}

In the previous section we have seen that the computation of
abstractions basically reduces to the over-approximation of attainable sets
of the plant.
A large number of over-approximation
methods have been proposed which apply to different classes of
systems, e.g.
\cite{Junge99,Osipenko07,Tabuada09,i11abs,i13qsupp}.
In this section, we present an approach to over-approximate attainable
sets of continuous-time perturbed control systems,
based on a matrix-valued Lipschitz inequality. 

\subsection{The sampled system}
 
Let us consider a perturbed control system of the form
\begin{equation}
\label{e:System:c-time}
\dot x \in f(x,u) + W
\end{equation}
with $f:\mathbb{R}^n\times U\to \mathbb{R}^n$, $U\subseteq
\mathbb{R}^m$ and $W\subseteq \mathbb{R}^n$.
We assume throughout this section that
$U$ is non-empty, $W$ contains the origin, and that $f(\cdot,u)$ is
locally Lipschitz for all $u \in U$. 
We use the set $W$ to represent various uncertainties in the
dynamics of the control system \ref{e:System:c-time}.

For $\tau \in \mathbb{R}_+$ and an interval $I \subseteq \intcc{0,\tau}$, a \begriff{solution of~\ref{e:System:c-time} on $I$  with (constant) input $u \in U$} is
defined as an absolutely continuous function $\xi \colon I \to \mathbb{R}^n$ that satisfies
$\dot \xi(t) \in f(\xi(t),u) + W$ for almost every (a.e.)
$t \in I$. 
We say that $\xi$ is \begriff{continuable to $\intcc{0,\tau}$} if there exists a solution $\bar \xi$ of~\ref{e:System:c-time} on $\intcc{0,\tau}$ with input $u \in U$ such that $\bar \xi|_{I} = \xi$.

We formulate a sampled variant of~\ref{e:System:c-time} as system as follows.

\begin{definition}
\label{def:SampledSystem}
Let
$S_1$
be a simple system
of the form \ref{e:systemMooreId},
and let $\tau > 0$.
We say that $S_1$ is
the \begriff{sampled system} associated with the
control system \ref{e:System:c-time} and the \begriff{sampling time}
$\tau$, if $X_1=\mathbb{R}^n$, $U_1=U$ and
the following holds:
$x_1 \in F_1(x_0,u)$ iff there exist a solution $\xi$ of \ref{e:System:c-time} on $\intcc{0,\tau}$ with input $u$ satisfying $\xi(0) = x_0$
and $\xi(\tau) = x_1$.
\end{definition}

In the sequel, $\varphi$ denotes the general solution of the
unperturbed system associated with \ref{e:System:c-time}
for constant inputs.
That is, if $x_0 \in \mathbb{R}^n$, $u \in U$, and $f(\cdot,u)$ is
locally Lipschitz, then $\varphi(\cdot,x_0,u)$ is the unique
non-continuable solution of the
initial value problem $\dot x = f(x,u)$, $x(0) = x_0$
\cite{Hartman02}.

Similar to other
approaches~\cite{ZamaniPolaMazoTabuada10,RunggerStursberg12} to
over-approximate attainable sets that are known for \emph{unperturbed} systems, our 
computation of attainable sets of the perturbed system is based on an estimate of the distance of neighboring
solutions of~\ref{e:System:c-time}.

\begin{definition}
\label{def:GrowthBound}
Consider the sets $K\subseteq \mathbb{R}^n$, $U'\subseteq U$ and the
sampling time $\tau>0$.
A map $\beta \colon \mathbb{R}_{+}^n \times U' \to \mathbb{R}_{+}^n$
is a \begriff{growth bound} on $K$, $U'$ associated with $\tau$ and
\ref{e:System:c-time} if the following conditions hold:
\begin{enumerate}
\item
\label{def:GrowthBound:beta}
$\beta(r,u) \geq \beta(r',u)$ whenever $r \geq r'$ and $u \in U'$,
\item
\label{def:GrowthBound:bound}
$\intcc{0,\tau} \times K \times U' \subseteq \dom \varphi$ and
if $\xi$ is a solution of \ref{e:System:c-time} on $\intcc{0,\tau}$ with input $u \in U'$ and $\xi(0),p \in K$
then
\begin{equation}
\label{e:def:GrowthBound}
| \xi(\tau) - \varphi(\tau,p,u) | \leq \beta( | \xi(0) - p |, u)
\end{equation}
holds component-wise.
\popQED
\end{enumerate}
\end{definition}

Let us emphasize some distinct features of the estimate~\ref{e:def:GrowthBound}. First of all, we formulate the
inequality~\ref{e:def:GrowthBound} component-wise, which allows to
bound the difference of neighboring solutions for each state coordinate independently. Second, $\beta$ is
a local estimate, i.e., we require~\ref{e:def:GrowthBound} to hold
only for initial states in $K$.  Moreover, $\beta$ is allowed to
depend on the input, but these inputs are assumed to be constant, and
we do not bound the effect of different inputs on
the distance of the solutions.
All those properties
contribute to more accurate over-approximations of the attainable
sets. This, in turn, leads to less conservative abstractions;
see our example in Section \ref{ss:Example:ZamaniPolaMazoTabuada10}.
Note that it is also immediate to account for extensions like time
varying inputs and using different sampling times.

\subsection{The abstraction}
\label{ss:TheAbstraction}
We continue with the construction of an abstraction $S_2$ of the
sampled system $S_1$. 
The state alphabet $X_2$ of the abstraction is defined
as a cover of the state alphabet $X_1$ where the elements of the cover
$X_2$ are non-empty, closed \begriff{hyper-intervals}, i.e., every element $x_2\in X_2$
takes the form
\[
\segcc{a,b}
=
\mathbb{R}^n
\cap
\left(
\intcc{a_1,b_1} \times \cdots \times \intcc{a_n,b_n}
\right)
\]
for some $a,b \in ( \mathbb{R} \cup \{ \pm \infty \} )^n$, $a \le b$.

Our notion of hyper-intervals allows for unbounded cells in $X_2$.
Nevertheless, in the computation of the abstraction
$S_2$, we work with a subset $\bar X_2 \subseteq X_2$ of compact cells.
We interpret the cells in $\bar X_2$ as the ``real'' quantizer symbols,
and the remaining ones, as overflow symbols,
see~\cite[Sect.~III.A]{i11abs}.

\begin{definition}\label{d:abstraction}
Consider two simple systems $S_1$ and $S_2$ of the form~\ref{e:systemMooreId}, a set
$\bar X_2\subseteq X_2$ and
a function $\beta \colon \mathbb{R}_{+}^n \times U_2 \to \mathbb{R}_{+}^n$.
Given $\tau>0$, suppose that $S_1$ is
the sampled system associated with~\ref{e:System:c-time}
and sampling time $\tau$.
We call $S_2$ an \begriff{abstraction} of $S_1$
based on $\bar X_2$ and $\beta$, if 
\begin{enumerate}
  \item \label{d:abstraction:cover}
  $X_2$ is a cover of $X_1$ by non-empty, closed hyper-intervals
  and every element $x_2\in \bar X_2$ is compact;
  \item \label{d:abstraction:inputs}
  $U_2\subseteq U_1$;
  \item \label{d:abstraction:rhs:2}
  for $x_2\in \bar X_2$, $x'_2\in X_2$ and $u\in U_2$ we have
  \begin{IEEEeqnarray}{c}\label{e:abs:tf}
    \left( \varphi(\tau,c,u) + \segcc{-r',r'} \right) \cap
    x'_2\neq\emptyset \implies x'_2\in F_2(x_2,u),
    \IEEEeqnarraynumspace
  \end{IEEEeqnarray}
where $\segcc{a,b} = x_2$, $c = \tfrac{b+a}{2}$, $r = \tfrac{b-a}{2}$ and
  $r'=\beta(r,u)$;
  \item \label{d:abstraction:rhs:1}
  $F_2(x_2,u)=\emptyset$ whenever $x_2\in X_2\setminus \bar X_2$, $u\in U_2$.
\popQED
\end{enumerate}
\end{definition}

Note that the implicit
definition of the transition function $F_2$ according to
\ref{d:abstraction:rhs:2} in Definition \ref{d:abstraction} is
equivalently expressible as follows.
Let $u \in U_2$ and
$\segcc{a,b} \in \bar X_2$, then
$\segcc{a',b'} \in X_2$ has to be an element of $F_2(\segcc{a,b},u)$ if
\begin{equation*}
\label{e:th:abstraction}
a' - r'
\leq
\varphi(\tau,c,u)
\leq
b' + r'
\end{equation*}
holds, where $c$, $r$ and $r'$ are as in Definition
\ref{d:abstraction}.

We illustrate the transition function $F_2(x_2,u)$ of an abstraction
in~\ref{f:abs-rhs}.

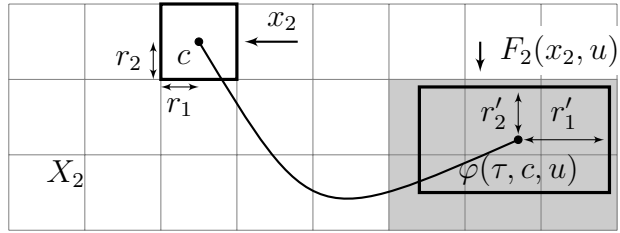
\begin{figure}[h]
\centering
\begin{tikzpicture}[>=latex']

\draw[fill,gray!40!white] (5,0) rectangle (8,2);

\draw[step=1cm,gray,very thin] (0,0) grid (8,3);

\draw[fill] (2.5,2.5) circle (.05) ;
\draw[very thick] (2,2) rectangle (3,3);
\draw[<->,thin] (1.9,2) -- node[left] {$r_2$} (1.9,2.5);
\draw[<->,thin] (2,1.9) -- node[below] {$r_1$} (2.5,1.9);

\draw[fill] (6.7,1.2) circle (.05);
\node at  (6.7,.8) {$\varphi(\tau,c,u)$};
\draw[very thick] (5.4,.5) rectangle (7.9,1.9);
\draw[<->,thin, shorten <=2pt, shorten >=2pt] (6.7,1.2) -- node[left] {$r_2'$}  (6.7,1.9);
\draw[<->,thin, shorten <=2pt, shorten >=2pt] (6.7,1.2) -- node[above] {$r_1'$}  (7.9,1.2);

\draw[thick] (2.5,2.5) .. controls (4,0)..  (6.7,1.2);

\node at (.75,.75) {$X_2$};
\node at (2.3,2.3) {$c$};
\draw[->,thick, shorten >=1mm] (3.8,2.5) -- node[near start,above] {$x_2$} (3,2.5);

\draw[->,thick, shorten >=1mm] (6.2,2.5) -- node[near start,right] {\colorbox{white}{$F_2(x_2,u)$}} (6.2,2);

\end{tikzpicture}
\caption{Illustration of the transition function of an abstraction.}
\label{f:abs-rhs}
\end{figure}

\begin{theorem}
\label{th:abstraction}
Consider two simple systems $S_1$ and $S_2$ of the form~\ref{e:systemMooreId}
and a set $\bar X_2\subseteq X_2$, and let $\tau>0$.
Suppose that $S_1$ is the sampled system associated
with~\ref{e:System:c-time} and sampling time $\tau$.
Let $\beta$ be a growth bound on $\cup_{x_2\in \bar
X_2}x_2$, $U_2$ associated
with $\tau$ and~\ref{e:System:c-time}.
If $S_2$ is 
an abstraction of $S_1$ based on
$\bar X_2$ and $\beta$, then $S_1 \preccurlyeq_{\in} S_2$.
\end{theorem}

\begin{proof}
To verify the condition \ref{prop:CanonicalFRR:in} in Proposition
\ref{prop:CanonicalFRR} first note that
$U_{S_2} ( x_2 ) = \emptyset$ if $x_2 \in X_2 \setminus \bar X_2$ by
our assumption on $S_2$. On the other hand, if $x_1 \in x_2 \in \bar
X_2$,
then $U_2 \subseteq U_{S_1}( x_1 )$ by our assumption on $\beta$, so the
condition \ref{prop:CanonicalFRR:in} in Proposition
\ref{prop:CanonicalFRR} is satisfied. To verify the
requirement \ref{prop:CanonicalFRR:dyn} in Proposition
\ref{prop:CanonicalFRR}, assume that
$x_2, x_2' \in X_2$ and $u\in U_{S_2}(x_2)$.
Then $x_2 \in \bar X_2$ by our assumption on $S_2$, so
$x_2 = \segcc{c-r,c+r}$ for some $c$, $r$. Moreover, if
additionally
$x_1 \in x_2$ and
$x_2' \cap F_1(x_1,u) \not= \emptyset$, then by
Definition \ref{def:SampledSystem} there exists a solution
$\xi \colon \intcc{0,\tau} \to \mathbb{R}^n$ of the system
\ref{e:System:c-time} with input $u$ satisfying $\xi(0) = x_1$ and
$\xi(\tau) \in x_2'$. It follows that $| \xi(0) - c | \leq r$, and hence,
$| \xi(\tau) - \varphi(\tau,c,u) | \leq r'$.
Then \ref{e:abs:tf} implies that $x_2' \in F_2(x_2,u)$.
An application of Proposition \ref{prop:CanonicalFRR} completes
the proof.
\end{proof}

As seen from the above proof, the set
$\varphi(\tau,c,u) + \segcc{-r',r'}$ in \ref{e:abs:tf}
over-approximates the attainable set $F_1( \segcc{a,b}, u)$. The
approximation error, which greatly influences the accuracy of the
abstraction, can be reduced by working with smaller cells
$\segcc{a,b}$. However, the accuracy can also
be improved without rediscretizing the state space $X_1$,
by covering cells $\segcc{a,b} \in \bar X_2$ by compact
hyper-intervals $\gamma_i + \segcc{-\rho_i,\rho_i}$ with
$\rho_i < r$, $i \in I$, and then using, in place of the premise in
\ref{e:abs:tf}, the test
$\exists_{i\in I}\left(
\varphi(\tau,\gamma_i,u) + \segcc{-\beta(\rho_i,u),\beta(\rho_i,u)}
\right)
\cap x_2' \not= \emptyset$.

\subsection{A growth bound}
\label{ss:GrowthBound}
In this subsection we present a specific growth bound
for the case that $f$ is continuously
differentiable in its first argument and the perturbations are given by $W=\segcc{-w,w}$ for
some $w\in \mathbb{R}_+^n$.
In the following proposition, we use $D_jf_i$ to denote the partial
derivative with respect to the $j$th
component of the first argument of $f_i$. 

\begin{theorem}
\label{prop:GrowthBound}
Let $\tau>0$ and let $f$, $U$ and $W$ be as in \ref{e:System:c-time}
with $W=\segcc{-w,w}$ for some $w \in \mathbb{R}^n_+$.
Let $U' \subseteq U$ and
assume in addition that $f(\cdot,u)$ is continuously differentiable
for every $u \in U'$.
Furthermore, let $K \subseteq K' \subseteq \mathbb{R}^n$ with
$K'$ being convex,
so that for any $u \in U'$, any $\tau ' \in \intcc{0,\tau}$ and any solution $\xi$ on $\intcc{0,\tau'}$
of \ref{e:System:c-time} with input $u$ and $\xi(0) \in K$,
we have $\xi(t) \in K'$ for all $t \in \intcc{0,\tau'}$.
Lastly, let the parametrized matrix
$L \colon U' \to \mathbb{R}^{n \times n}$ satisfy
\[
L_{i,j}(u)
\geq
\begin{cases}
D_j f_i(x,u),& \text{if $i=j$,}\\
| D_j f_i(x,u) |,& \text{otherwise}
\end{cases}
\]
for all $x \in K'$ and all $u \in U'$. Then
any $\xi$ as above is continuable to $\intcc{0,\tau}$, and
the map $\beta$ given by
\begin{equation*}
\label{e:GrowthBound}
\beta(r,u)
=
\e^{L(u)\tau}
\;r
+
\int_0^\tau
\e^{L(u)s}
\;w
\;\mathrm{d}s
\end{equation*}
is a growth bound on $K$, $U'$ associated with $\tau$ and \ref{e:System:c-time}.
\end{theorem}

Theorem \ref{prop:GrowthBound} can be applied quite easily for obtaining growth bounds. 
Firstly, the computation of an a priori enclosure $K'$ to solutions of
\ref{e:System:c-time}
is standard, e.g. \cite{NedialkovJacksonCorliss99} and the references therein. 
Secondly, the parametrized matrix $L$ requires bounding partial derivatives on $K'$. 
Such bounds can be computed in an automated way using, e.g., interval arithmetic \cite{Moore66}. 
Finally, given $L$, the evaluation of the expression for $\beta$ is straightforward. 
We emphasize, however, that Theorem \ref{prop:GrowthBound} provides
only one of several methods to over-approximate attainable sets. Any
over-approximation method can be used to compute abstractions based on
feedback refinement relations.
\looseness-1

Having a growth bound at hand, the application of Theorem
\ref{th:abstraction} becomes a routine task. Examples are presented in
the next section.

For the proof of Theorem \ref{prop:GrowthBound} we need the following
auxiliary result, which appears in \cite{Walter64ENGLISH} without proof.

\begin{lemma}
\label{l:auxforgrowthbound}
Let $\tau>0$ and
$A \subseteq \mathbb{R}^n$.
Let $\xi_i \colon \intcc{0,\tau} \to A$, $i\in\{1,2\}$, be
two perturbed solutions of a dynamical system
with continuous right hand side $f \colon \mathbb{R}^n \to \mathbb{R}^n$, i.e.,
the maps $\xi_i$ are absolutely continuous and satisfy
\begin{equation*}
|\dot \xi_i(t) - f(\xi_i(t))| \leq w_i(t) \quad \text{for a.e.} \quad t \in \intcc{0,\tau},
\end{equation*}
where $w_i \colon \intcc{0,\tau} \to \mathbb{R}^n_+$, $i \in \{1,2\}$,
are integrable. Consider a matrix $L \in \mathbb{R}^{n\times n}$ with $L_{i,j} \geq 0$ for $i \neq j$ and suppose that for all $x,y\in A$ we have
\begin{equation}
\label{l:auxforgrowthbound:hy:5}
x_i \geq y_i \implies f_i(x) - f_i(y) \leq \sum_{j=1}^n\nolimits L_{i,j}|x_j-y_j|.
\end{equation}
Let $\rho \colon \intcc{0,\tau} \to \mathbb{R}^n_+$ be absolutely continuous and satisfying
\begin{equation*}
\label{l:auxforgrowthbound:hy:4}
\dot \rho (t) = L \rho(t) + w_1(t) +w_2(t)
\end{equation*}
for a.e. $t \in \intcc{0,\tau}$. Then $|\xi_1(0) - \xi_2(0)| \leq \rho(0)$ implies \mbox{$|\xi_1 (t) - \xi_2(t) | \leq \rho(t)$} for every $t \in \intcc{0,\tau}$.
\end{lemma}
\begin{proof}
Let $\tilde \rho \colon \intcc{0,\tau}\to \mathbb{R}^n_+$ be absolutely
continuous such that $\tilde \rho(0) = \rho(0)$ and
$\tilde \rho'(t) = L \tilde \rho(t) + w_1(t)+ w_2(t) + \varepsilon$
for some $\varepsilon \in (\mathbb{R}_+ \setminus \{0\})^n$ and
a.e. $t \in \intcc{0,\tau}$. We shall prove that
\begin{equation}
\label{l:auxforgrowthbound:e:2GR}
| \xi_1(t) - \xi_2(t) | \leq \tilde \rho(t)
\end{equation}
holds for all $t \in \intcc{0,\tau}$, so that the lemma follows from a
limit argument. To this end, denote the function
$\xi_1 - \xi_2 - \tilde \rho$ on $\intcc{0,\tau}$ by $z$ and let
$t_0 = \sup\{ t \in \intcc{0,\tau}\ | \ \forall_{s \in \intcc{0,t}}\ z(s) \leq 0\}$.
Then $t_0 \geq 0$ as $|\xi_1(0)-\xi_2(0)|\leq \rho(0)$, and
since we can interchange the roles of $\xi_1$ and $\xi_2$ if necessary, we
may assume without loss of generality that
\ref{l:auxforgrowthbound:e:2GR} holds for all $t \in
\intcc{0,t_0}$. It remains to show that $t_0 = \tau$.

Assume that $t_0 < \tau$. Using \ref{l:auxforgrowthbound:e:2GR}, a
continuity argument shows that
we may choose $t_2 \in \intoc{t_0,\tau}$ and $i \in \intcc{1;n}$ such that $z_i(t_2) > 0$, $z_i(t_0) = 0$ and 
\begin{equation}
\label{l:auxforgrowthbound:e:2}
\varepsilon_i + \sum_{j=1}^n\nolimits L_{i,j} \tilde \rho_j(t) \geq
\sum_{j=1}^n\nolimits L_{i,j} | \xi_{1,j}(t) - \xi_{2,j}(t)|
\end{equation}
for all $t \in \intcc{t_0,t_2}$. 
Define $t_1 = \sup\{t \in \intcc{t_0,t_2} \ | \  z_i(t) \leq 0 \}$ 
and note that $z_i(t_1) = 0$ as $z_i$ is continuous. 
The inequality 
$z_i'(t) \leq f_i(\xi_1(t)) - f_i(\xi_2(t)) + w_{1,i}(t)+w_{2,i}(t) - \tilde \rho'_i(t)$ 
for a.e. $t \in \intcc{t_1,t_2}$ and
the definition of $\tilde \rho$
then imply that
\begin{align*}
\label{e:2}
z_i(t_2) \leq \int_{t_1}^{t_2}\! \big ( f_i(\xi_1(t)) - f_i(\xi_2(t)) - 
\sum_{j=1}^n L_{i,j} \tilde \rho_j(t) - \varepsilon_i  \big ) \mathrm{d}t.
\end{align*}
Thus, $z_i(t_2) \leq 0$ by \ref{l:auxforgrowthbound:e:2} and \ref{l:auxforgrowthbound:hy:5}. 
This contradicts our choice of $t_2$, and so $t_0 = \tau$. 
\end{proof}

\begin{proof}[Proof of Theorem \ref{prop:GrowthBound}]
\label{pageref:review:item12}
Fix $p \in K$, $u \in U'$ and note that $\beta(r,u)\geq \beta(r',u)$ if $r\geq r'$ as all
entries of $\e^{L(u)\tau}$ are non-negative \cite[Th.~7.7]{Kato82}.
Next, we show
that condition \ref{def:GrowthBound:bound} in Definition \ref{def:GrowthBound} holds.
In order to apply Lemma \ref{l:auxforgrowthbound}
we shall establish \ref{l:auxforgrowthbound:hy:5}
for
$K'$, $f(\cdot,u)$ and $L(u)$ in place of $A$, $f$ and $L$.
Indeed, by the mean value theorem,
there exists $z \in \{ x + t(y-x) | t \in \intcc{0,1} \}$
such that $f_i(x,u) -f_i(y,u) = \sum_{j=1}^n D_jf_i(z,u)(x_j-y_j)$.
Hence, by the definition of $L$, we obtain \ref{l:auxforgrowthbound:hy:5}.
Now, let
$\xi$ be a solution on $\intcc{0,\tau}$ of \ref{e:System:c-time}
with input $u$ such that $\xi(0) \in K$. By Filippov's Lemma \cite{Filippov62}, there
exists an integrable map
$s \colon \intcc{0,\tau} \to W$
such
that $\dot \xi(t) = f(\xi(t),u) + s(t)$ for a.e. $t \in \intcc{0,\tau}$.
So, apply Lemma \ref{l:auxforgrowthbound} to $f(\cdot,u)$, $K'$, $\varphi(\cdot,p,u)$, $\xi$, $0$, $w$ and $L(u)$
in place of $f$, $A$, $\xi_1$, $\xi_2$, $w_1$, $w_2$ and $L$,
respectively, to obtain
$|\xi(\tau)-\varphi(\tau,p,u)|\leq \beta(|\xi(0)-p|,u)$.\\
Finally, suppose there exists $\xi \colon \intcc{0,\tau'} \to K'$ as in the statement of the theorem that is not continuable to $\intcc{0,\tau}$.
Then, there exist $t_0 \in \intcc{0,\tau}$ and a solution $\bar \xi \colon \intco{0,t_0} \to \mathbb{R}^n$ of \ref{e:System:c-time} with input $u$ such that
$\bar \xi|_{\intcc{0,\tau'}}=\xi$ and
$\bar \xi(t)$ becomes unbounded as $t \in \intco{0,t_0}$ approaches $t_0$ \cite{Filippov88}.
On the other hand,
applying Lemma \ref{l:auxforgrowthbound} to $f(\cdot,u)$, $K'$, $\bar \xi|_{\intcc{0,t}}$, $\xi(0)$, $w$, $|f(\xi(0),u)|$, $L(u)$ and $t$
in place of $f$, $A$, $\xi_1$, $\xi_2$, $w_1$, $w_2$, $L$ and $\tau$
we conclude that $|\bar \xi(t) - \xi(0)|$ is uniformly bounded for $t \in \intco{0,t_0}$, which is
a contradiction.
\end{proof}

\subsection{The Case of Periodic Dynamics}
\label{ss:s:review:PeriodicDynamics}

Occasionally we will have to consider continuous-time control systems
of the form \ref{e:System:c-time} whose dynamics are periodic, i.e.,
$f(\xi + p,\cdot) = f(\xi,\cdot)$ for some \begriff{period}
$p \in \mathbb{R}^n \setminus \{0\}$ and all $\xi \in \mathbb{R}^n$.
Our result below shows how to exploit periodicity
to obtain abstractions
that are finite and yet are capable of
reproducing solutions that are unbounded in the direction of the
period.
This is useful, e.g. when one of the components of the state
represents an angle and the number of full loops is potentially
unbounded; see Section \ref{ss:Example:ZamaniPolaMazoTabuada10} for an example.

\begin{theorem}
\label{th:periods}
Let $p_1,\ldots,p_\ell \in \mathbb{R}^n$,
$\ell \in \mathbb{N}$,
be such that 
$f$ in \ref{e:System:c-time} satisfies $f(x+p_i,u) = f(x,u)$ for 
all $i \in \intcc{1;\ell}$, $x \in \mathbb{R}^n$ and
$u \in U$.
Consider systems $S_1$ and $S_2$ of the form~\ref{e:systemMooreId}, 
where $U_2 \subseteq U_1$ and
$S_1$ is the sampled system associated with \ref{e:System:c-time} and 
sampling time $\tau > 0$.
Define the map $\pi \colon X_1 \rightrightarrows X_1$ by
$\pi(x)
=
\Menge{ x + \sum_{i=1}^\ell k_i p_i }{ k \in \mathbb{Z}^\ell }
$,
and let $R$ be a set of non-empty subsets of $X_1$ such that
$X_2 = \Menge{ \pi(\Omega) }{ \Omega \in R }$ and $X_2$ is
a cover of $X_1$.\\
Then $S_1 \preccurlyeq_\in S_2$ iff the following conditions hold:
\begin{asparaenum}[(a)]
\item
\label{th:periods:in}
$x \in \Omega \in R$ implies
$U_{S_2}(\pi(\Omega)) \subseteq U_{S_1}(x)$.
\item
\label{th:periods:dyn}
If
$\Omega, \Omega' \in R$,
$u\in U_{S_2}(\pi(\Omega))$ and
$\pi(\Omega') \cap F_1(\Omega,u) \not= \emptyset$,
then $\pi(\Omega') \in F_2(\pi(\Omega),u)$.
\popQED
\end{asparaenum}
\end{theorem}

Obviously, the transition function $F_2$ of the system $S_2$ can be
computed by over-approximating attainable sets $F_1(\Omega,u)$ as
detailed in Sections \ref{ss:Canonical}, \ref{ss:TheAbstraction} and
\ref{ss:GrowthBound}, and by verifying the condition
$(\Omega' + \sum_{i=1}^\ell k_i p_i) \cap F_1(\Omega,u) \not= \emptyset$,
for $\Omega, \Omega' \in R$ with $\Omega$ being compact, and
finitely many $k \in \mathbb{Z}^\ell$.

\begin{proof}
First observe that
$F_1(x,u) + \innerProd{k}{p} = F_1(x + \innerProd{k}{p}, u)$
for all $k \in \mathbb{Z}^\ell$, $x \in X_1$ and $u \in U_1$,
where $\innerProd{k}{p} = \sum_{i=1}^\ell k_i p_i$.
Then $U_{S_1}(x + \innerProd{k}{p}) = U_{S_1}(x)$ for all $x \in X_1$
and all $k \in \mathbb{Z}^\ell$, which shows that the condition
\ref{prop:CanonicalFRR:in} in Proposition \ref{prop:CanonicalFRR} is
equivalent to \ref{th:periods:in}.
We shall show that the condition \ref{prop:CanonicalFRR:dyn} is
equivalent to \ref{th:periods:dyn}, which proves the theorem.
If $\Omega, \Omega' \in R$, $u \in U_{S_2}(\pi(\Omega))$ and
$\pi(\Omega') \cap F_1(\Omega,u) \not= \emptyset$, then
$\pi(\Omega), \pi(\Omega') \in X_2$ and $\Omega \subseteq \pi(\Omega)$,
and so \ref{prop:CanonicalFRR:dyn} shows that
$\pi(\Omega') \in F_2(\pi(\Omega),u)$.
Conversely, if $\Omega, \Omega' \in X_2$,
$u \in U_{S_2}(\Omega)$ and
$\Omega' \cap F_1(\Omega,u) \not= \emptyset$, then there exist
$\Omega_0, \Omega_0' \in R$ satisfying $\Omega = \pi(\Omega_0)$ and
$\Omega' = \pi(\Omega_0')$. Hence,
$\pi(\Omega_0') \cap (\innerProd{k}{p} + F_1(\Omega_0,u)) \not= \emptyset$
for some $k \in \mathbb{Z}^\ell$, and since
$\pi(\Omega_0') = \pi(\Omega_0') + \innerProd{k}{p}$ we have
$\pi(\Omega_0') \cap F_1(\Omega_0,u) \not= \emptyset$.
Then \ref{th:periods:dyn} shows that
$\Omega' \in F_2(\Omega,u)$,
which completes the proof.
\end{proof}

\section{Examples}
\label{s:ex}

In this section, we
demonstrate the practicality of our approach
on
control problems for nonlinear plants.

\subsection{A path planning problem for an autonomous vehicle}
\label{ss:Example:ZamaniPolaMazoTabuada10}

\label{ex:Zamani}
We consider
an autonomous vehicle whose dynamics we assume to be 
given by the bicycle model in \cite[Ch.~2.4]{AstromMurray08}. More 
concretely,
the dynamics of the system
are of the form \ref{e:System:c-time},
where \mbox{$f \colon \mathbb{R}^3 \times U \to \mathbb{R}^3$} is given by 
\begin{equation*}
f(x,(u_1,u_2)) = \begin{pmatrix}
u_1\cos(\alpha + x_3) \cos(\alpha)^{-1} \\
u_1\sin(\alpha + x_3) \cos(\alpha)^{-1} \\
u_1\tan(u_2)
\end{pmatrix}
\end{equation*}
with $U = \intcc{-1,1} \times \intcc{-1,1}$
and $\alpha = \arctan(\tan(u_2)/2)$. 
Here, $(x_1,x_2)$ is the position
and $x_3$ is the orientation of the vehicle in the $2$-dimensional plane.
The control inputs $u_1$ and $u_2$ are the rear wheel velocity 
and the steering angle.
Perturbations are not acting on the system dynamics,
i.e., $W = \{(0,0,0)\}$.

The concrete control problem is formulated with respect to the sampled system
$S_1$  associated with \ref{e:System:c-time} and sampling time $\tau=0.3$. The
control objective is to enforce a certain patrolling behavior on the vehicle which is situated in a maze; see~\ref{fig:zamani}.
Specifically, the vehicle, 
whose initial state is $A_{1,0} = \{(0.4,0.4,0)\}$, 
should patrol infinitely often between the target regions 
$A_{1,\mathrm{r}_1} = \intcc{0,0.5} \times \intcc{0,0.5} \times \mathbb{R}$ and
$A_{1,\mathrm{r}_2} = (9,0,0) + A_{1,\mathrm{r}_1}$,
while avoiding the
obstacles $A_{1,\mathrm{a}}$.
The third 
component of $A_{1,\mathrm{a}}$ equals
$\mathbb{R}$. We formalize our concrete control problem through the pair $(S_1,\Sigma_1)$ with the specification $\Sigma_1$ defined as
\begin{IEEEeqnarray}{c}\label{ex:robot:spec}
\begin{IEEEeqnarraybox}[][c]{l}
\{
(u,x)\in (U_1\times X_1)^{\mathbb{Z}_+}\mid x(0)\in A_{1,0}\implies
\\
\quad
\forall_{t\in \mathbb{Z}_+}
(
x(t) \notin A_{1,\mathrm{a}}
\land
\forall_{i\in \{1,2\}}
\exists_{t' \in \intco{t;\infty}}
x(t') \in A_{1,\mathrm{r}_i}
)
\},
\end{IEEEeqnarraybox}
\IEEEeqnarraynumspace
\end{IEEEeqnarray}
where $U_1=U$ and $X_1 = \mathbb{R}^3$.
To solve $(S_1,\Sigma_1)$ we solve an abstract control problem
$(S_2,\Sigma_2)$ as detailed below.
As $f$ possesses the period $p=(0,0,2\pi)$ we construct
a canonical abstraction $S_2$ of the form \ref{e:systemMooreId}
using Theorem \ref{th:periods}, where
$R$ consist of the
shifted copies of the hyper-interval
\[
\intcc{-\tfrac{1}{10},\tfrac{1}{10}}\times \intcc{-\tfrac{1}{10},\tfrac{1}{10}} \times \intcc{-\tfrac{\pi}{35},\tfrac{\pi}{35}},
\]
whose centers form the set
$
\tfrac{2}{10}\intcc{0;50} \times \tfrac{2}{10}\intcc{0;50} \times \tfrac{2\pi}{35} \intcc{-17;17}
$,
and of the hyper-intervals 
$\{x \in \mathbb{R}^3 \mid x_j \geq 10.1\}$, 
$\{x \in \mathbb{R}^3 \mid x_j \leq -0.1\}$, 
$j \in \{1,2\}$.
Set
$U_2 = \{0,\pm 0.3,\pm 0.6,\pm 0.9\} \times \{0,\pm 0.3,\pm 0.6,\pm 0.9\}$,
and let $X_2$ be as in Theorem \ref{th:periods}. 
The transition function $F_2$ is computed according to the remark
following Theorem \ref{th:periods}, in which
$F_2(x_2,u) = \emptyset$ 
if $(x_2,u) \in X_2 \times U_2$, $x_2 \cap A_{1,\mathrm{a}} \neq \emptyset$. 
The required growth bound $\beta$ on
$\mathbb{R}^3$, $U_2$ associated with $\tau$ and \ref{e:System:c-time}
is obtained using Theorem \ref{prop:GrowthBound}. 
In particular, $\beta(r,u)=\e^{L(u)\tau}r$, 
where $L$ is given by
$L_{1,3}(u_1,u_2) = L_{2,3}(u_1,u_2) = | u_1 \sqrt{\tan^2(u_2)/4 + 1 } |$, 
and $L_{i,j}(u_1,u_2) = 0$ for $(i,j) \notin \{(1,3),(2,3)\}$.

The computation of $F_2$ takes $2.25$ seconds (Intel Core i7 2.9 GHz) 
resulting in an abstraction having $37266181$ transitions.

To construct the abstract specification $\Sigma_2$ we let 
$A_{2,0}=\{x_2\in X_2\mid x_2\cap A_{1,0}\neq \emptyset\}$,
$A_{2,\mathrm{r}_i}=\{x_2\in X_2\mid x_2\subseteq A_{1,\mathrm{r}_i}\}$, $i \in \{1,2\}$ and
$A_{2,\mathrm{a}}=\{x_2\in X_2\mid x_2\cap A_{1,\mathrm{a}}\neq \emptyset\}$;
see \ref{fig:zamani}.
We define $\Sigma_2$
by~\ref{ex:robot:spec}, where we 
substitute $U_1$, $X_1$, $A_{1,0}$, $A_{1,\mathrm{r}_1}$, $A_{1,\mathrm{r}_2}$, $A_{1,\mathrm{a}}$
with $U_2$, $X_2$, $A_{2,0}$, $A_{2,\mathrm{r}_1}$, $A_{2,\mathrm{r}_2}$, $A_{2,\mathrm{a}}$,
respectively. It is straightforward to verify that $\Sigma_2$ is an
abstract specification associated with $S_1$, $S_2$, $\in$ and
$\Sigma_1$.

The abstract problem $(S_2,\Sigma_2)$ can be solved using the algorithm in
\cite[Fig.~1]{BloemJobstmanPitermanPnueliSaar12}, which
simplifies here to two rather than three nested fixed-point iterations
since for our problem the \begriff{general reactivity (1)} specification in
\cite{BloemJobstmanPitermanPnueliSaar12}
reduces to
$\nu Z.\cap_{i \in \{1,2\}} \mu
Y.(\cox Y\cup (A_{2,\mathrm{r}_i}\cap \cox Z))$, where
$\cox A
=
\Menge{x \in X_2}{ \exists_{u \in U_{2}} \emptyset \not= F_2(x,u) \subseteq A }$.
We actually use a Dijkstra-like algorithm
\cite{GalloLongoPallotinoNguyen93} for the inner fixed-point to
successfully solve $(S_2,\Sigma_2)$ within $0.54$ seconds.
The solution is refined to a solution of
$(S_1,\Sigma_1)$ by adding a static quantizer; see Theorem \ref{th:refinement}.
A similar problem with considerably
less complex specification is solved in
\cite{ZamaniPolaMazoTabuada10}, where the run times in seconds are
$13509$ (abstraction) and $535$ (synthesis) on Intel Core 2 Duo 2.4 GHz.

We would like to discuss two of the advantages of the growth bounds
we have introduced in Section \ref{s:Construction}.
As we already mentioned, $\beta$ bounds each component of neighboring solutions separately,
which can be directly seen by the formula
$\beta(r,u) = r + r_3 \! \cdot L_{1,3}(u_1,u_2) \cdot \! (\tau,\tau,0)^{\top}$.
This distinguishes $\beta$ from an estimate based on a norm.
Moreover, $\beta$ depends on the input,
which is crucial for the present example. Indeed,
the function $\e^{(\sup L)\tau}r$, where
$\sup L \in \mathbb{R}^{3 \times 3}$ is given by $(\sup L)_{i,j} =
\sup_{u \in U_2} L_{i,j}(u)$, is also a growth bound on
$\mathbb{R}^3$, $U_2$ associated with $\tau$ and
\ref{e:System:c-time},
which leads to an abstraction with $43288873$ transitions.
However, due to the poor approximation quality of this growth bound we obtain
an unsolvable abstract control problem.
\begin{figure}
\centering
\psfrag{11}[r][]{\footnotesize$0$}
\psfrag{12}[r][]{}
\psfrag{13}[r][]{}
\psfrag{14}[r][]{}
\psfrag{15}[r][]{}
\psfrag{16}[r][]{\footnotesize$5$}
\psfrag{17}[r][]{}
\psfrag{18}[r][]{}
\psfrag{19}[r][]{}
\psfrag{20}[r][]{}
\psfrag{21}[r][]{\footnotesize$10$}
\psfrag{0}[t][]{\footnotesize$0$}
\psfrag{1}[t][]{}
\psfrag{2}[t][]{}
\psfrag{3}[t][]{}
\psfrag{4}[t][]{}
\psfrag{5}[t][]{\footnotesize$5$}
\psfrag{6}[t][]{}
\psfrag{7}[t][]{}
\psfrag{8}[t][]{}
\psfrag{9}[t][]{}
\psfrag{10}[t][]{\footnotesize$10$}
\includegraphics[width=0.95\linewidth]{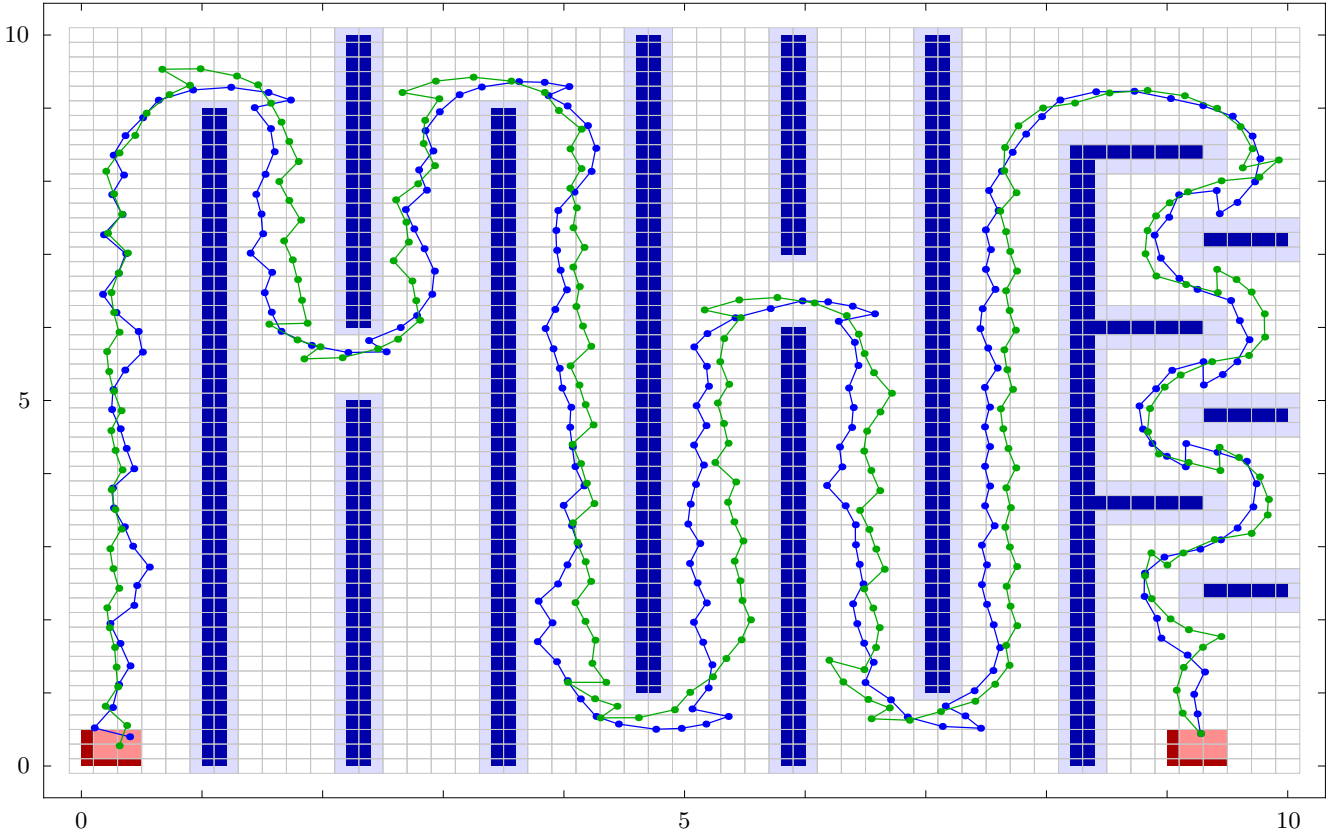}
\caption{\label{fig:zamani} Projection of the states of $S_1$ and
  $S_2$
  to $\mathbb{R}^2 \times \{0\}$. The sets
  $A_{1,\mathrm{a}}$ and $A_{1,\mathrm{r}_1}$, $A_{1,\mathrm{r}_2}$
  are indicated in dark blue and in red, respectively. The states in
  $A_{2,\mathrm{a}}$ and $A_{2,\mathrm{r}_1}$,$A_{2,\mathrm{r}_2}$ are
  indicated in blue and in light red, respectively. 
A closed-loop trajectory of the concrete control problem is shown 
evolving from $A_{1,\mathrm{r}_1}$ to $A_{1,\mathrm{r}_2}$ 
in the blue part and vice versa in the green part.}
\end{figure}

\subsection{An aircraft landing maneuver}
\label{ex:aircraft}

We consider an aircraft DC9-30 whose dynamics we model
according to
\cite{MitchellBayenTomlin01}.
We use $x_1,x_2,x_3$ to denote the
state variables, which respectively correspond to the velocity, the flight path angle and the altitude of the
aircraft. The input alphabet is given by $U =
\intcc{0,160\cdot 10^3} \times \intcc{0^\circ,10^\circ}$ 
and represents the thrust of the engines (in Newton) and the angle of attack. 
The dynamics are given by $f\colon \mathbb{R}^3 \times U \to \mathbb{R}^3$,
\begin{equation*}
f(x,u)= \begin{pmatrix}
\frac{1}{m} ( u_1 \cos u_2 - D(u_2,x_1) - m g \sin x_2) \\ 
\frac{1}{mx_1} ( u_1 \sin u_2 + L(u_2,x_1) - m g \cos x_2 ) \\
x_1 \sin x_2
\end{pmatrix},
\end{equation*}
where $D(u_2,x_1)= (2.7 + 3.08\cdot (1.25 + 4.2\cdot u_2)^2)\cdot x_1^2$, $L(u_2,x_1) = (68.6\cdot (1.25 + 4.2\cdot u_2))\cdot x_1^2$ and $mg = 60 \cdot 10^3 \cdot 9.81$ account for the drag, lift and gravity, respectively \cite{MitchellBayenTomlin01}.

We consider the input disturbance $P_1\colon U \rightrightarrows U$ 
given by 
$P_1(u)= (u + \intcc{-5\cdot 10^3,5 \cdot 10^3} \times \intcc{-0.25^\circ,0.25^\circ})\cap U$ 
and measurement errors of the form 
$P_2 \colon \mathbb{R}^3 \rightrightarrows \mathbb{R}^3$ given by
$P_2(x) = x + \frac{1}{20}\intcc{-0.25,0.25} \times
\frac{1}{20}\intcc{-0.05^\circ,0.05^\circ} \times \frac{1}{20}\intcc{-1,1}$. We do not
consider
any further disturbances,
i.e.,
we let
$W=\{(0,0,0)\}$, $P_3=\id$,
and
$P_4 = \id$. 

The concrete control problem
is formulated
with respect to the sampled system $S_1 = ( X_1, X_1, U_1, U_1, X_1, F_1, \id )$
associated with \ref{e:System:c-time}
and the sampling time $\tau=0.25$.
We aim at steering
the aircraft from an altitude of $55$ meters close to the ground with an appropriate total
and horizontal touchdown velocity.
More formally, the specification $\Sigma_1$ is given by
\begin{IEEEeqnarray}{c}\label{ex:aircraft:spec}
\begin{IEEEeqnarraybox}[][c]{l}
\Sigma_1
=
\big \{(u,x)\in  (U_1 \times X_1)^{\mathbb{Z}_+}   \ | \ x(0) \in A_{0} \ \implies \\
\qquad \qquad \quad (\exists_{s\in \mathbb{Z}_+} \ x(s) \in A_{\mathrm{r}} \ \wedge \ \forall_{t \in \intco{0;s} } \  x(t) \notin A_{\mathrm{a}}  ) \big \},
\end{IEEEeqnarraybox}
\IEEEeqnarraynumspace
\end{IEEEeqnarray}
where $I = \intcc{-3^\circ,0^\circ}$,
$A_{0} = \intcc{80,82} \times \intcc{-2^\circ,-1^\circ} \times \{55\}$,
\begin{align*}
A_{\mathrm{a}} &= \mathbb{R}^3 \setminus (\intcc{58,83} \times I \times \intcc{0,56}), \\
A_{\mathrm{r}} &= (\intcc{63,75} \times I \times \intcc{0,2.5}) \cap
\{x \in \mathbb{R}^3 | x_1 \sin x_2 \geq -0.91 \}.
\end{align*}
As detailed in Section \ref{ss:Uncertainties}, the perturbed control
problem is solved through an auxiliary unperturbed control problem.
To begin with, define the simple system $\hat S_1$
by \ref{e:pert:sys:aux} with $\hat U_1=U$.
\label{review:item11:text1}
Next, let $X$ be a cover
of $\mathbb{R}^3$ formed by subdividing $\mathbb{R}^3 \setminus A_{\mathrm{a}}$
into {$210 \cdot 210 \cdot 210$} hyper-intervals, and suitable unbounded hyper-intervals.
Define $X_2 = \{  P_2^{-1}(\Omega) \ | \ \Omega \in X\}$
and let $\bar
X_2$ be the subset of compact
elements of $X_2$ that do not intersect $A_{\mathrm{a}}$. Define the abstraction for $\hat S_1$
as the simple system $S_2$ given by \ref{e:systemMooreId}, 
where $U_2 = \{0,32000\} \times U'$, $U'$ contains precisely 10 inputs
equally spaced in $\intcc{0^\circ,8^\circ}$. We apply Theorem \ref{prop:GrowthBound}
with $w = M (5000,0.25^\circ)^\top \leq (0.108,0.002,0)^\top$ and a suitable a priori enclosure $K'$ to obtain a growth
bound, where $M \in \mathbb{R}_+^{2\times 3}$ satisfies
$M_{i,j} \geq |D_{j,2}f_i(x,u)|$
for all $x\in K'$ and $u \in P_1(U_2)$. 
Here, $D_{j,2}f_i$ stands for the partial derivative with respect to the $j$th component of the second argument of $f_i$. 
Note that $w$ accounts for the perturbation $P_1$. Then, we use
Theorem \ref{th:abstraction} to compute $F_2$ such that
\mbox{$\hat S_1 \preccurlyeq_{\in} S_2$}.  
The computation takes $674$ seconds resulting in an abstraction 
with about $9.38 \cdot 10^9$ transitions (Intel Xeon E5 3.1 GHz). 

To construct the abstract specification $\Sigma_2$ for $S_2$ we let 
$A_{2,0} = \{x_2 \in X_2 \ | \ x_2 \cap A_{1,0} \neq \emptyset\}$,
$A_{2,\mathrm{a}} = \{x_2 \in X_2 \ | \ x_2 \cap A_{1,\mathrm{a}} \neq \emptyset\}$,
$A_{2,\mathrm{r}} = \{x_2 \in X_2 \ | \ x_2 \subseteq A_{1,\mathrm{r}}\}$ and 
define the specification $\Sigma_2$ by \ref{ex:aircraft:spec} with 
$U_2$, $X_2$, $A_{2,0}$, $A_{2,\mathrm{r}}$, $A_{2,\mathrm{a}}$ in place of 
$U_1$, $X_1$, $A_{0}$, $A_{\mathrm{r}}$, $A_{\mathrm{a}}$. 
It is easy to verify that $\Sigma_2$ is an abstract
specification associated with $\hat S_1$, $S_2$, $\in$ and $\Sigma_1$. Note that 
$\Sigma_2$ (as well as $\Sigma_1$) is a particular instance of a \begriff{reach-avoid} specification.
Using a standard technique \cite{GalloLongoPallotinoNguyen93}, the
abstract control problem $(S_2,\Sigma_2)$
is successfully solved within $26$ seconds.
By Corollary \ref{c:PerturbedRefinement}
the behavior of the perturbed closed loop is a subset of $\Sigma_1$. 
See \ref{fig:aircraft:trajectory}.

We proceed to make some comments on solving perturbed control problems. 
At first, Theorem \ref{prop:GrowthBound} allows to deal with time-varying input perturbations,
when the theorem is applied as in this example. 
Second, accounting for measurement errors only requires inflating the
cells that would have been used if 
measurement errors were not present. To conclude, perturbed control problems 
can be solved in our framework by using canonical abstractions.

\begin{figure}[ht]
\psfrag{470}[r][]{\small 10}
\psfrag{940}[r][]{\small 20}
\psfrag{1410}[r][]{\small 30}
\psfrag{1880}[r][]{\small 40}
\psfrag{2350}[r][]{\small 50}
\psfrag{0}[t][b]{\small 0}
\psfrag{1680}[t][b]{\small 5}
\psfrag{3360}[t][b]{\small 10}
\psfrag{5040}[t][b]{\small 15}
\psfrag{6720}[t][b]{\small 20}
\psfrag{y}[b][t]{\small $x_3$}
\psfrag{x}[l][]{\small $t$}
\begin{center}
\includegraphics[scale=0.47]{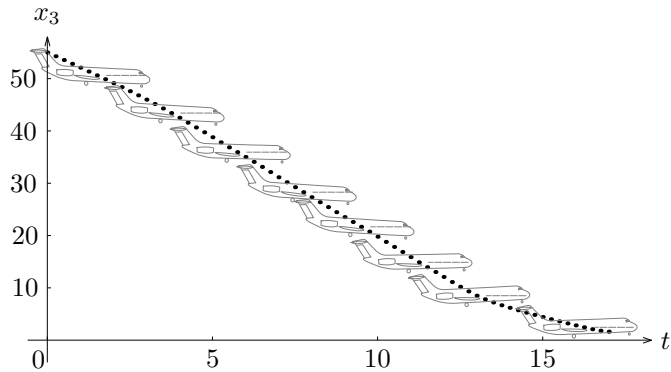}
\end{center}
\caption{\label{fig:aircraft:trajectory}Time evolution of the altitude of the aircraft in the closed loop.
The aircraft pitch $u_2 + x_2$ is indicated for $8$ instants of time.}
\end{figure}

\section{Conclusions}
\label{s:Conclusions}

We have presented a novel approach to abstraction-based controller
synthesis which builds on the concept of feedback refinement relation
introduced in the present paper. Our framework incorporates several
distinct features. Foremost, the designed controllers require quantized
(or symbolic) state information only and are connected to the plant
via a static quantizer, which is particularly important
for any practical implementation of the controller.
Our work permits the synthesis of robust
correct-by-design controllers in the presence of various
uncertainties and disturbances,
and more generally, applies to a broader class of
synthesis problems
than previous research addressing the state information and refinement
complexity issues as explained and illustrated in Sections
\ref{s:intro} and \ref{s:pitfalls}.
Moreover,
we do not assume that the controller is able to set the initial state of
the plant, which is
also important
in the context of practical control systems.

We have additionally identified a class of canonical abstractions, and
have presented a method to compute such abstractions for perturbed
nonlinear control systems.
We utilized numerical examples to
demonstrate the  applicability and efficiency of our
synthesis framework. We emphasize, however, that the computational
effort is still expected to grow rapidly with the dimension of the
state space of the plant, a problem that is shared by all grid based
methods for the computation of abstractions.

\section*{Acknowledgment}

\noindent
The authors thank
M.~Mazo (Delft),
T.~Moor (Erlangen),
P.~Tabuada (Los Angeles)
and M.~Zamani (M\"unchen)
for fruitful discussions about this research.

\bibliographystyle{IEEEtran}
\bibliography{GR/IEEEtranBSTCTL,GR/preambles,GR/mrabbrev,GR/strings,GR/fremde,GR/eigeneCONF,GR/eigeneJOURNALS,GR/eigenePATENT,GR/eigeneREPORTS,GR/eigeneTALKS,GR/eigeneTHESES}
\end{document}